\documentclass[preprint,12pt,numbers,sort&compress]{elsarticle}
\usepackage{soul}
\usepackage{psfrag}
\usepackage{a4wide}
\usepackage{rotating}
\usepackage{color}
\usepackage{amssymb}
\usepackage{amsmath}
\usepackage{amsthm}
\usepackage{verbatim}
\usepackage{tikz}
\newcommand{\f}[1]{\boldsymbol{#1}}

\newcommand{\R}{\mathbb R}

\newcommand{\I}{\mathbb I}
\newcommand{\US}[3]{\mathcal{S}_{#3}^{#1,#2}}
\newcommand{\UN}[3]{N_{#3}^{#1,#2}}
\newcommand{\UM}[3]{M_{#3}^{#1,#2}}
\newcommand{\UR}[3]{R_{#3}^{#1,#2}}
\newcommand{\V}{\mathcal{V}}
\newcommand{\vol}{\text{vol}}
\newcommand{\grad}{\text{grad}}

\newcommand{\indexPatch}{\mathcal{I}_{\Omega}}
\newcommand{\indexFace}{\mathcal{I}_{\Gamma}}
\newcommand{\indexFaceBoundary}{\mathcal{I}_{\Gamma_B}}
\newcommand{\indexFaceInner}{\mathcal{I}_{\Gamma_I}}
\newcommand{\indexEdge}{\mathcal{I}_{\Sigma}}

\newcommand{\indexVertex}{\mathcal{I}_{\Xi}}

\newcommand{\patch}{\Omega}
\newcommand{\face}{\Gamma}
\newcommand{\edge}{\Sigma}
\newcommand{\vertex}{\Xi}
\newcommand{\map}{\f{F}}
\newcommand{\domain}{\Omega}
\newcommand{\surf}{\f{G}}

\newcommand{\phiPatch}[4]{\phi_{\Omega^{(#1)};#2,#3,#4}}
\newcommand{\phiFace}[4]{\phi_{\Gamma^{(#1)};#2,#3,#4}}
\newcommand{\phiOneEdgeAll}[1]{\psi_{\Sigma^{(#1)}}}
\newcommand{\phiOneVertexAll}[1]{\psi_{\Xi^{(#1)}}}
\newcommand{\phiFaceAll}[1]{\phi_{\Gamma^{(#1)}}}
\newcommand{\phiFaceAllRes}[1]{\phi_{\Gamma^{(#1)};B}}
\newcommand{\phiFaceAllResMod}[1]{\widetilde{\phi}_{\Gamma^{(#1)};B}}

\newcommand{\phiEdgeAllSingle}[1]{\phi_{\Sigma;#1}}
\newcommand{\phiEdgeAllSingleMod}[1]{\widetilde{\phi}_{\Sigma;#1}}

\newcommand{\fPatch}[5]{f^{(#5)}_{\Omega^{(#1)};#2,#3,#4}}
\newcommand{\fFace}[5]{f^{(#5)}_{\Gamma^{(#1)};#2,#3,#4}}

\newcommand{\fstandard}[1]{f^{(#1)}}

\newcommand{\ftrace}{f_0}
\newcommand{\fder}{f_1}

\newcommand{\coeff}{a}
\newcommand{\coeffAllFace}{\f{a}_{\Gamma}}
\newcommand{\coeffAllEdge}{\f{a}_{\Sigma}}
\newcommand{\coeffAllEdgeMod}{\widetilde{\f{a}}_{\Sigma}}
\newcommand{\coeffPatchVertex}[5]{a^{(#5)}_{\Xi^{(#1)};#2,#3,#4}}

\newcommand{\coeffPatchEdgeSt}[5]{a^{(#5)}_{\Sigma^{(#1)};#2,#3,#4}}
\newcommand{\coeffFace}[4]{a_{\Gamma^{(#1)};#2,#3,#4}}

\newcommand{\TMatrixFace}{T_{\Gamma}}
\newcommand{\TMatrixEdge}{T_{\Sigma}}
\newcommand{\TMatrixEdgeMod}{\widetilde{T}_{\Sigma}}

\newcommand{\SpacePatch}[1]{\V^1_{\Omega^{(#1)}}}
\newcommand{\SpaceFace}[1]{\V^{1}_{\Gamma^{(#1)}}}

\newcommand{\SpaceEdgeAll}{\V^{1}_{\Sigma}}
\newcommand{\SpaceFaceAll}{\V^{1}_{\Gamma}}
\newcommand{\SpacePatchTwo}[1]{\widetilde{\V}^1_{\Omega^{(#1)}}}
\newcommand{\SpaceFaceTwo}[1]{\widetilde{\V}^{1}_{\Gamma^{(#1)}}}
\newcommand{\SpacePatchMod}[1]{\widetilde{\V}^1_{\Omega^{(#1)}}}
\newcommand{\SpaceFaceMod}[1]{\widetilde{\V}^{1}_{\Gamma^{(#1)}}}
\newcommand{\SpaceEdgeAllMod}{\widetilde{\V}^{1}_{\Sigma}}

\newcommand{\WhichPatchesE}[1]{E_{#1}}
\newcommand{\WhichPatchesV}[1]{V_{#1}}
\newcommand{\WhichEdgesV}[1]{D_{#1}}
\newcommand{\LeftFace}[2]{L_{#1,#2}}
\newcommand{\RightFace}[2]{R_{#1,#2}}

\newcommand{\numPatches}{\nu}

\newcommand{\classVolume}{\mathcal{A}}

\DeclareMathOperator{\Span}{span}

\newtheorem{thm}{Theorem}

\newtheorem{ass}{Assumption}

\theoremstyle{definition}

\newtheorem{rem}{Remark}
\advance\textheight by 0.4cm
\advance\topmargin by -0.2cm

\definecolor{gold}{rgb}{1,0.7,0}
\definecolor{dred}{rgb}{0.92,0,0}
\definecolor{dgreen}{rgb}{0,0.6,0}

\definecolor{myred}{rgb}{1,0.2,0.2}
\definecolor{mygreen}{rgb}{0,0.5,0}
\definecolor{myblue}{rgb}{0.2,0.2,1}
\definecolor{mypurple}{rgb}{1,0.3,1}
\definecolor{myviolet}{rgb}{0.3,0,0.7}

\begin{document}

\begin{frontmatter}

\title{$C^1$ isogeometric spline space for trilinearly parameterized multi-patch volumes}

\author[vil]{Mario Kapl\corref{cor}}
\ead{m.kapl@fh-kaernten.at}

\author[slo1]{Vito Vitrih}
\ead{vito.vitrih@upr.si}
 
\address[vil]{Department of Engineering $\&$ IT, Carinthia University of Applied Sciences, Villach, Austria}

\address[slo1]{IAM and FAMNIT, University of Primorska, Koper, Slovenia}

\cortext[cor]{Corresponding author}

\begin{abstract}
We study the space of $C^1$ isogeometric spline functions defined on trilinearly parameterized multi-patch volumes. Amongst others, we present a general framework for the design of the $C^1$ isogeometric spline space and of an associated basis, which is based on the 
two-patch construction~\cite{BiKa19}, and which works uniformly for any possible multi-patch configuration. The presented method is demonstrated in more detail on the basis of a particular subclass of trilinear multi-patch volumes, namely for the class of trilinearly parameterized multi-patch volumes with exactly one inner edge. For this specific subclass of trivariate multi-patch parameterizations, we further numerically compute the dimension of the resulting $C^1$ isogeometric spline space and use the constructed $C^1$ isogeometric basis functions to numerically explore the approximation properties of the $C^1$ spline space by performing $L^2$ approximation.
\end{abstract}

\begin{keyword}
isogeometric analysis \sep $C^1$-continuity \sep geometric continuity \sep multi-patch volume \sep isogeometric basis functions
\MSC[2010] 65N35 \sep 65D17 \sep 68U07
\end{keyword}

\end{frontmatter}

%%%%%%%%%%%%%%%%%%%%%%%%%%%%%%%%%%%%%%%%%%%%%%%%%%%%%%%%%%%%%%%%%%%%%
\section{Introduction}
%%%%%%%%%%%%%%%%%%%%%%%%%%%%%%%%%%%%%%%%%%%%%%%%%%%%%%%%%%%%%%%%%%%%%

In the context of isogeometric analysis (IgA)~\cite{ANU:9260759, CottrellBook, HuCoBa04}, the construction of globally $C^1$~spline spaces over 
multi-patch geometries is a topic of great interest, since it allows the solving of fourth order partial differential equations (PDEs) over complex 
geometries just via their weak form and a standard Galerkin discretization, see e.g. \cite{BaDe15,CoSaTa16,KaBuBeJu16,NgKaPe15,TaDe14} for the biharmonic 
equation, \cite{ABBLRS-stream,benson2011large,kiendl-bazilevs-hsu-wuechner-bletzinger-10,kiendl-bletzinger-linhard-09,KiHsWuRe15} for the Kirchhoff-Love 
shell problem, \cite{gomez2008isogeometric,GoCaHu09,LiDeEvBoHu13} for the Cahn-Hilliard equation, and \cite{gradientElast2011,MaReBeJu18,KhakaloNiiranenC1} for 
problems of strain gradient elasticity. While in case of bivariate multi-patch geometries, i.e. for planar multi-patch domains and multi-patch surfaces, 
the design of $C^1$ spline spaces has been intensively studied in the last years, cf. the recent survey paper~\cite{KaSaTa19b}, in case of trivariate multi-patch 
parameterizations, that is for multi-patch volumes, the construction of such smooth spline spaces has been dealt with just in a small number of publications, 
see e.g. \cite{BiJuMa17,BiKa19,BiJuMa20,ChAnRa19,NgKaPe15}.

The work in~\cite{NgKaPe15} is based on a sweeping approach, which generates from a $C^1$ bivariate multi-patch spline space a particular $C^1$ spline 
space over a multi-patch volume. The constructed isogeometric functions are triquadratic except in the vicinity of an extraordinary vertex or edge, i.e 
a vertex or edge with a valency different to four, where the degree is raised to be three or four depending on the valency of the vertex or edge. 
In~\cite{ChAnRa19}, for a given general muti-patch volume with possibly extraordinary vertices or edges, partial degree elevation across the common faces is 
performed to construct $C^1$ spline spaces with good approximation properties. The proposed method relies on solving a large homogeneous system of linear 
equations, is applicable to any spline degree~$p \geq 2$, and results in isogeometric functions with an increased degree in neighborhood of the 
common faces of the multi-patch volume.

The publications~\cite{BiJuMa17,BiKa19,BiJuMa20} explore the entire $C^1$ isogeometric spline space over trilinearly parameterized two-patch volumes. While 
in \cite{BiJuMa17} the dimension and a basis of the $C^1$ spline space is numerically obtained for a spline degree~$p=3$ and 
$p=4$, in \cite{BiKa19}, a full theoretical framework for any spline degree~$p \geq 3$ is developed to compute the dimension and to generate a basis of the 
$C^1$~spline space. In contrast to~\cite{ChAnRa19,NgKaPe15}, the constructed $C^1$ isogeometric basis functions possess for both cases~\cite{BiJuMa17,BiKa19} 
the same degree on the whole multi-patch volume. Furthermore, the $C^1$ basis construction~\cite{BiJuMa17} is used in~\cite{BiJuMa20} for the case of
$p=3$ and $p=4$ to perform $L^2$ projection on trilinearly parameterized two-patch volumes to numerically investigate the approximation properties of the 
corresponding $C^1$ isogeometric spline space. The numerical results indicate an optimal approximation power for $p=4$ and a slightly reduced one 
for $p=3$ which is mainly effected by the reduced convergence caused in the vicinity of the common face of the two-patch domain. 

Another approach for the design of smooth spline spaces over multi-patch volumes, which is related to the problem of constructing $C^1$ isogeometric 
multi-patch spline spaces but does not completely solve the issue, is the technique~\cite{WeiBlendedBC}. There, a tricubic spline space over a given 
multi-patch volume is generated, which is $C^2$ inside the single patches, $C^1$ across the common faces but just $C^0$ at extraordinary edges and vertices.

This paper extends now the work in~\cite{BiJuMa17,BiKa19} for the construction of $C^1$ isogeometric spline spaces over trilinearly parameterized two-patch 
volumes to the case of trilinear multi-patch volumes with arbitrary many patches and with possibly extraordinary edges and vertices. More precisely, we 
analyze the space of $C^1$ isogeometric spline functions defined on these trilinear multi-patch volumes and describe a general framework to generate
the $C^1$ isogeometric spline space and a basis of the space. The proposed technique relies on the constructed $C^1$~basis functions 
for the two-patch case in~\cite{BiKa19} and can be applied in a uniform way to any possible trilinear multi-patch volume and to any spline degree~$p \geq 3$. 
In addition, a specific subclass of trilinearly parameterized multi-patch volumes, namely the class of trilinear multi-patch volumes with exactly one 
inner edge, is considered in more detail. This subclass of trivariate multi-patch parameterizations is of particular interest since it comprises 
multi-patch volumes with still a small number of patches but which already allow to model quite complex domains. For this specific subclass of 
trilinear multi-patch volumes, the dimension of the resulting $C^1$ isogeometric spline space is numerically computed, and $L^2$ approximation is performed to 
numerically study the approximation properties of the $C^1$ spline space.

The remainder of the paper is organized as follows. Section~\ref{sec:volumes-spaces} presents the class of trilinearly parameterized 
multi-patch volumes and introduces further the concept of $C^1$~isogeometric spline spaces over this class of multi-patch volumes. 
In Section~\ref{sec:C1-continuity}, the $C^1$~continuity condition across an interface of a given multi-patch volume is discussed, which has been already 
studied in~\cite{BiKa19} for the case of a trilinear two-patch volume, and which can be used to represent a $C^1$ isogeometric spline 
function in an explicit form in the vicinity of the interface. This representation is then employed in Section~\ref{sec:design_study_C1spaces} to 
develop a general framework for the design of the $C^1$ isogeometric spline space over the given trilinear multi-patch volume and of an associated
basis of the space. For a particular subclass of trilinear multi-patch volumes, namely for the class of trilinearly parameterized multi-patch volumes 
with exactly one inner edge, the basis construction is discussed in more detail in Section~\ref{sec:specific-class}, and is used to numerically compute 
the dimension of the associated $C^1$ isogeometric spline space and to numerically investigate the approximation power of the space. Finally, we conclude 
the paper in Section~\ref{sec:conclusion}.

%%%%%%%%%%%%%%%%%%%%%%%%%%%%%%%%%%%%%%%%%%%%%%%%%%%%%%%%%%%%%%%%%%%%%%%%%%%%%%%%%%%%%%%%%%%%%%%%%%%%%%%
\section{Multi-patch volumes \& $C^1$~isogeometric spline spaces} \label{sec:volumes-spaces}
%%%%%%%%%%%%%%%%%%%%%%%%%%%%%%%%%%%%%%%%%%%%%%%%%%%%%%%%%%%%%%%%%%%%%%%%%%%%%%%%%%%%%%%%%%%%%%%%%%%%%%

We will first introduce the particular setting of the volumetric multi-patch domains, which will be used throughout the paper, and then shortly present 
the concept of $C^1$~isogeometric spline spaces over these domains.

%%%%%%%%%%%%%%%%%%%%%%%%%%%%%%%%%%%%%%%%%%%%%%%%%%%%%%%%%%%%%%%%%%%%%
\subsection{Trilinear multi-patch volumes}
%%%%%%%%%%%%%%%%%%%%%%%%%%%%%%%%%%%%%%%%%%%%%%%%%%%%%%%%%%%%%%%%%%%%%%

Let $\domain \subset \R^3$ be an open domain, whose closure $\overline{\domain}$ is the disjoint union of open components given by 
hexahedral patches~$\patch^{(i)}$, $i \in \indexPatch$, quadrilateral faces~$\face^{(i)}$, $i \in \indexFace$, edges~$\edge^{(i)}$, 
$i \in \indexEdge$, and vertices~$\vertex^{(i)}$, $i \in \indexVertex$, i.e.
\begin{equation*} 
\displaystyle
\overline{\domain} = \bigcup_{i \in \indexPatch} \patch^{(i)}  \; \dot{\cup}  \bigcup_{i \in \indexFace} \face^{(i)} \; 
\dot{\cup} \bigcup_{i \in \indexEdge} \edge^{(i)} \dot{\cup} \bigcup_{i \in \indexVertex} \vertex^{(i)},
\end{equation*}
where the symbol $\dot{\cup}$ is used to denote the disjoint union of two sets.
Each vertex~$\vertex^{(i)}$, $i \in \indexVertex$, is a point in the space, i.e. $\vertex^{(i)} \in \R^3$, and each edge~$\edge^{(i)}$, 
$i \in \indexEdge$, face~$\face^{(i)}$, $i \in \indexFace$, and patch~$\patch^{(i)}$, $i\in \indexPatch$, is defined by $2$, $4$ or $8$ of 
these vertices via 
\begin{equation} \label{eq:par_edge}
 \edge^{(i)}= \{ (1-\xi)\vertex^{(i_0)} + \xi \vertex^{(i_1)} \; |\;  \xi \in (0,1) \},
\end{equation}
\begin{equation} \label{par_face}
\face^{(i)} = \{ (1-\xi_1)(1-\xi_2)\vertex^{(i_0)} + \xi_1(1-\xi_2) \vertex^{(i_1)} + (1-\xi_1)\xi_2\vertex^{(i_2)} + \xi_1\xi_2 \vertex^{(i_3)} \; | \; 
(\xi_1,\xi_2) \in (0,1)^2 \},
\end{equation}
or
\begin{eqnarray}
 \patch^{(i)} & = & \{ (1-\xi_1)(1-\xi_2)(1-\xi_3)\vertex^{(i_0)} + \xi_1(1-\xi_2)(1-\xi_3) \vertex^{(i_1)} + \nonumber \\
& & (1-\xi_1)\xi_2 (1-\xi_3)\vertex^{(i_2)} +
 \xi_1\xi_2 (1-\xi_3) \vertex^{(i_3)} +  (1-\xi_1)(1-\xi_2)\xi_3\vertex^{(i_4)} + \label{eq:par_patch} \\  
 & &\xi_1(1-\xi_2)\xi_3 \vertex^{(i_5)} + (1-\xi_1)\xi_2 \xi_3 \vertex^{(i_6)} 
+ \xi_1\xi_2 \xi_3 \vertex^{(i_7)} \; | \;  
 (\xi_1,\xi_2,\xi_3) \in (0,1)^{3} \}, \nonumber
\end{eqnarray}
respectively. The faces~$\face^{(i)}$, $i \in \indexFace$, will be distinguished throughout the paper between boundary faces~$\face^{(i)}$, 
$i \in \indexFaceBoundary$, i.e. $\face^{(i)} \subset \partial \domain$, and inner faces~$\face^{(i)}$, $i \in \indexFaceInner$, i.e. 
$\face^{(i)} \subset \overline{\domain} \setminus \partial \domain$, and we will further have $\indexFace = \indexFaceBoundary \dot{\cup} \indexFaceInner$.
We assume that the domain $\overline{\Omega}$ does not possess any hanging vertex or edge. Moreover, we assume that all parameterizations 
in~\eqref{eq:par_edge}--\eqref{eq:par_patch} are non-singular, and denote by $\map^{(i)}$ the 
extended trilinear and non-singular parameterization in~\eqref{eq:par_patch}, called \textit{geometry mapping}, for the closure of 
$\patch^{(i)}$, i.e., $\map^{(i)}: [0,1]^3 \rightarrow \overline{\patch^{(i)}}$. Clearly, each vertex~$\vertex^{(i)}$, $i \in \indexVertex$, each 
edge~$\edge^{(i)}$, $i \in \indexEdge$, and each face~$\face^{(i)}$, $i \in \indexFace$, can then be also interpreted as the image of a boundary 
point, open boundary edge, or open boundary face of the unit cube~$[0,1]^3$ for at least one geometry mapping~$\map^{(i)}$, $i \in \indexPatch$.

The domain~$\overline{\domain}$ is also referred to as \textit{multi-patch volume}, and the collection of all geometry mappings~$\map^{(i)}$, 
$i \in \indexPatch$, is called the \textit{geometry} of the multi-patch volume~$\overline{\domain}$. An example of a three-patch 
volume~$\overline{\domain}$ with their individual patches~$\patch^{(i)}$, $i \in \{1,2,3\}$, and with their corresponding geometry 
mappings~$\map^{(i)}$, as well as the decomposition of the three-patch volume into the single patches~$\patch^{(i)}$, $i \in \indexPatch$, 
faces~$\face^{(i)}$, $i \in \indexFace$, edges $\edge^{(i)}$, $i \in \indexEdge$, and vertices~$\vertex^{(i)}$, $i \in \indexVertex$, is shown in 
Fig.~\ref{fig:geometrymapping_volumes}. 
\begin{figure}[htp]
\begin{center}
\begin{tabular}{cc}
\includegraphics[width=9.4cm]{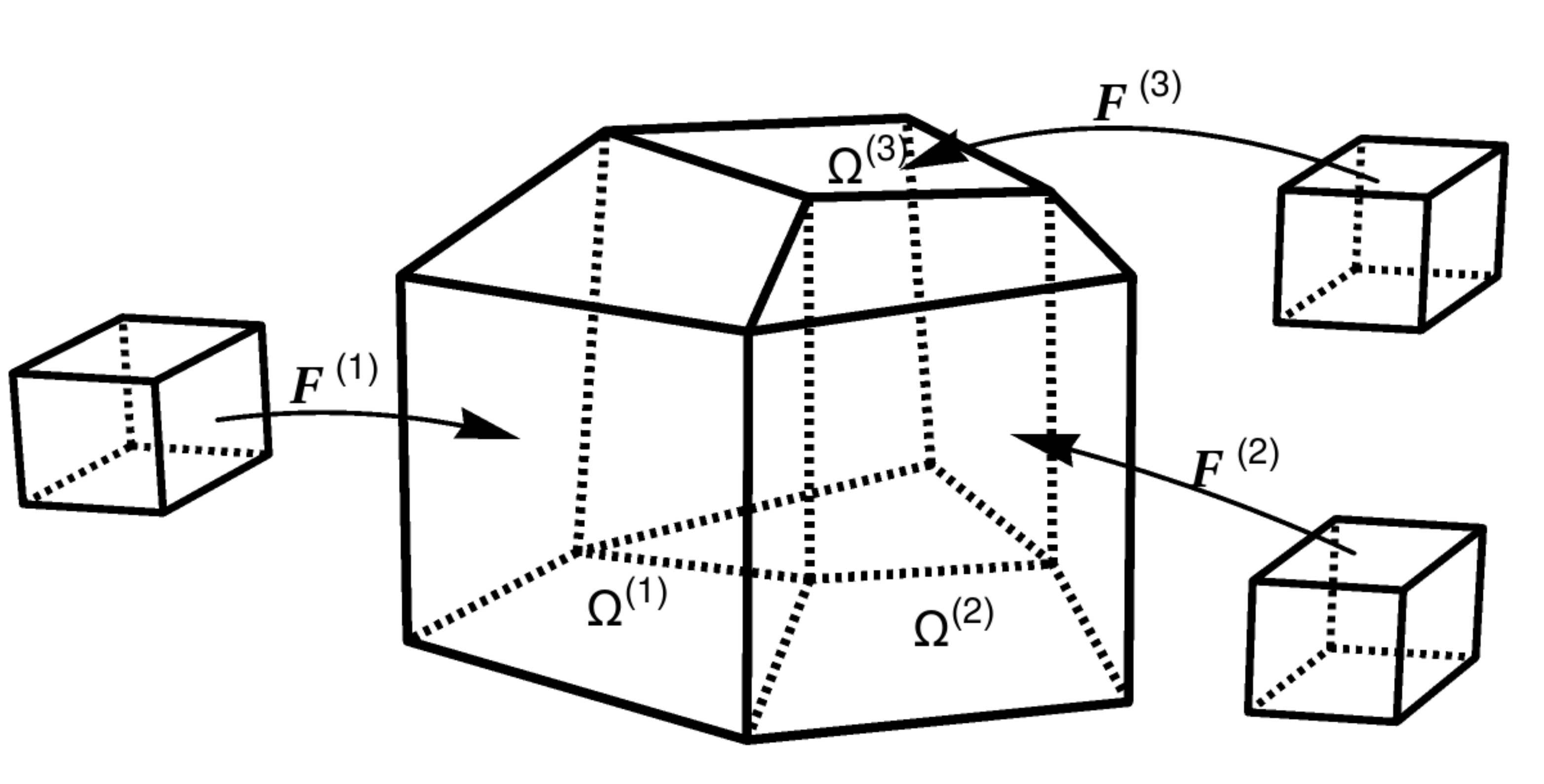}&
\includegraphics[width=5.9cm]{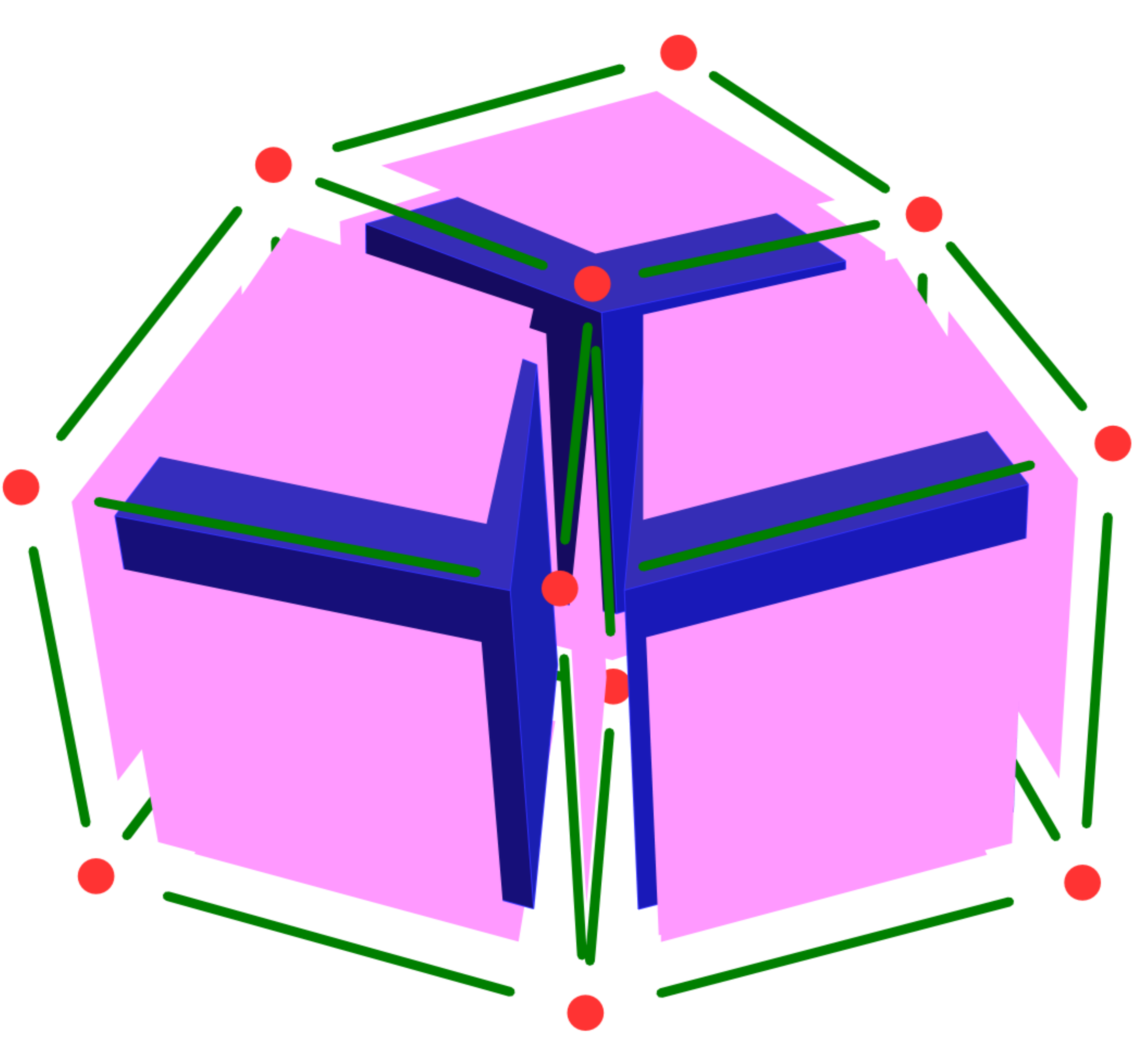}
\end{tabular}
\end{center}
\caption{Left: A trilinear multi-patch volume~$\overline{\domain}$ with three patches~$\patch^{(i)}$, $i \in \{1,2,3 \}$, and with their associated geometry 
mappings~$\map^{(i)}$. Right: The decomposition of the multi-patch volume~$\overline{\domain}$ into the single patches~$\patch^{(i)}$ (blue), 
faces~$\face^{(i)}$ (pink), edges~$\edge^{(i)}$ (green) and vertices $\vertex^{(i)}$ (red).} 
\label{fig:geometrymapping_volumes} 
\end{figure}

%%%%%%%%%%%%%%%%%%%%%%%%%%%%%%%%%%%%%%%%%%%%%%%%%%%%%%%%%%%%%%%%%%%%%%%%%%%%%%%%%%%%%%%%%%%%%%%%%%%%%%%%%%%%%%%%%%%%%%%%%
\subsection{The concept of $C^1$~isogeometric spline spaces for multi-patch volumes} \label{subsec:concept_C1spaces}
%%%%%%%%%%%%%%%%%%%%%%%%%%%%%%%%%%%%%%%%%%%%%%%%%%%%%%%%%%%%%%%%%%%%%%%%%%%%%%%%%%%%%%%%%%%%%%%%%%%%%%%%%%%%%%%%%%%%%%%%%%

Let $p \geq 3$, $1 \leq r \leq p-2$ and $k \geq 0$. We denote by $\US{p}{r}{h}$ the univariate spline space of degree~$p$ and 
regularity~$r$ on a uniform mesh of size $h=\frac{1}{k+1}$ over the unit interval~$[0,1]$, where $k+1$ is the number of spline elements. Moreover, 
let $\US{\f{p}}{\f{r}}{h}$ with $\f{p}=(p,p)$ and $\f{r}=(r,r)$ or $\f{p}=(p,p,p)$ and $\f{r}=(r,r,r)$ be the bivariate or trivariate tensor-product 
spline space~$\US{p}{r}{h} \times \US{p}{r}{h}$ or $\US{p}{r}{h} \times \US{p}{r}{h} \times \US{p}{r}{h}$ over the unit square~$[0,1]^2$ or the 
unit cube~$[0,1]^3$, respectively. We denote by $\UN{p}{r}{j}$, $j \in \I =\{0,1,\ldots, n-1 \}$, the  B-splines of the univariate spline 
space~$\US{p}{r}{h}$, where $n$ is the dimension of the spline space~$\US{p}{r}{h}$, i.e. $n= \dim \US{p}{r}{h} = p + 1 + k(p-r)$, 
and we denote by $\UN{\f{p}}{\f{r}}{j_1,j_2} =\UN{p}{r}{j_1} \UN{p}{r}{j_2}$, $j_1,j_2 \in \I$, and
$\UN{\f{p}}{\f{r}}{j_1,j_2,j_3} =\UN{p}{r}{j_1} \UN{p}{r}{j_2}\UN{p}{r}{j_3}$, $j_1,j_2,j_3 \in \I$, the tensor-product B-splines of the 
bivariate and trivariate spline space~$\US{\f{p}}{\f{r}}{h}$, respectively. 

Note that the trilinear parameterization~\eqref{eq:par_patch} of the geometry mappings~$\map^{(i)}$, $i \in \indexPatch$, trivially implies that 
$\map^{(i)} \in (\US{\f{p}}{\f{r}}{h})^3$. Then, the space of isogeometric spline functions on the multi-patch volume~$\overline{\domain}$ 
(with respect to the spline space~$\US{\f{p}}{\f{r}}{h}$) is defined as
\[
 \V = \left\{ \phi \in L^2(\overline{\domain}) \; | \; \phi |_{\overline{\patch^{(i)}}} \in \US{\f{p}}{\f{r}}{h} \circ (\map^{(i)})^{-1}, \; 
i \in \indexPatch   \right\}.
\]
Therefore, an isogeometric function~$\phi \in \V$ possesses for each patch~$\patch^{(i)}$, $i \in \indexPatch$, a spline 
function~$\fstandard{i} = \phi \circ \map^{(i)}$ of the form
\begin{equation*} \label{eq:standard_representation}
\fstandard{i}(\xi_1,\xi_2,\xi_3) = \sum_{j_1=0}^{n-1} \sum_{j_2=0}^{n-1} \sum_{j_3=0}^{n-1} \coeff_{j_1,j_2,j_3}^{(i)} \UN{\f{p}}{\f{r}}{j_1,j_2,j_3} 
(\xi_1,\xi_2,\xi_3), \quad (\xi_1,\xi_2,\xi_3) \in [0,1]^3, 
\end{equation*}
with $\coeff_{j_1,j_2,j_3}^{(i)} \in \R$. We are interested in the $C^1$ subspace of $\V$, which is given by 
\[
 \V^1 = \V \cap C^1(\overline{\domain}).
\]
The space~$\V^1$ is fully characterized by the observation that \textit{an isogeometric function}~$\phi \in \V$ \textit{belongs to the space} $\V^1$ \textit{if and only if for any two neighboring patches}~$\patch^{(i_0)}$ \textit{and} $\patch^{(i_1)}$, $i_0,i_1 \in \indexPatch$, \textit{with the common inner face}~$\overline{\face^{(i)}} = \overline{\patch^{(i_0)}} \cap \overline{\patch^{(i_1)}}$, $i \in \indexFaceInner$, \textit{the associated graph patches} 
$\left[\begin{array}{cc}
  \map^{(i_0)} &
  \phi \circ \map^{(i_0)}
 \end{array} \right]^T$
and 
$\left[\begin{array}{cc}
  \map^{(i_1)} & 
  \phi \circ \map^{(i_1)}
 \end{array} \right]^T$
\textit{are} $G^1$\textit{-continuous}, cf.~\cite{BiJuMa17,Pe15}. The $C^1$ continuity condition of an isogeometric function~$\phi \in \V$ 
has been studied in detail for the two-patch case in~\cite{BiKa19}, and will be recalled in the next section.

%%%%%%%%%%%%%%%%%%%%%%%%%%%%%%%%%%%%%%%%%%%%%%%%%%%%%%%%%%%%%%%%%%%%%
\section{$C^1$ continuity condition} \label{sec:C1-continuity}
%%%%%%%%%%%%%%%%%%%%%%%%%%%%%%%%%%%%%%%%%%%%%%%%%%%%%%%%%%%%%%%%%%%%%

We present the $C^1$ smoothness condition of an isogeometric function~$\phi \in \V$, which has been already considered for the case of a two-patch 
volume in~\cite{BiKa19}, and which can be also used in this work to study the $C^1$ isogeometric spline space~$\V^1$ and to generate a basis for 
the space~$\V^1$ over any possible configuration of a trilinear multi-patch volume~$\overline{\domain}$. Firstly, we will introduce some required 
assumptions, definitions and concepts.

%%%%%%%%%%%%%%%%%%%%%%%%%%%%%%%%%%%%%%%%%%%%%%%%%%%%%%%%%%%%%%%%%%%%%%%%
\subsection{Assumptions, definitions and concepts}
%%%%%%%%%%%%%%%%%%%%%%%%%%%%%%%%%%%%%%%%%%%%%%%%%%%%%%%%%%%%%%%%%%%%%%%%%

Let $\face^{(i)}$, $ i \in \indexFaceInner$, be an inner face with two neighboring patches~$\patch^{(i_0)}$ 
and $\patch^{(i_1)}$, $i_0,i_1 \in \indexPatch$, i.e. $\overline{\face^{(i)}} = \overline{\patch^{(i_0)}} \cap \overline{\patch^{(i_1)}}$. Then, we will 
assume throughout the paper and without loss of generality that the two associated geometry mappings~$\map^{(i_0)}$ and $\map^{(i_1)}$ are 
(re)parameterized in such a way that the common inner face~$\overline{\face^{(i)}}$ is given by
\begin{equation} \label{eq:standard_form_two_patches} 
 \map^{(i_0)}(0,t_1,t_2) = \map^{(i_1)}(t_1,0,t_2),  \quad
 (t_1,t_2) \in [0,1]^2, 
\end{equation}
and that the $12$ different vertices of the two-patch volume~$\overline{\patch^{(i_0)}} \cup \overline{\patch^{(i_1)}}$ are labeled as 
$\vertex^{(i_0)}, \vertex^{(i_1)},\ldots,$ $\vertex^{(i_{11})}$, cf. Figure~\ref{fig:standard_two-patch_case}. We also say that the two geometry mappings 
$\map^{(i_0)}$ and $\map^{(i_1)}$ are given in \textit{standard form} with respect to their common inner face~$\face^{(i)}$ when they 
fulfill~\eqref{eq:standard_form_two_patches}.

\begin{figure}
\begin{center}
\resizebox{0.75\textwidth}{!}{
 \begin{tikzpicture}
  \coordinate(K) at (-0.5,0); \coordinate(L) at (1.4,0.5); \coordinate(M) at (1.3,4); \coordinate(N) at (0.2,2.7);
  \fill[gray!10] (L) -- (M) -- (M) -- (N) -- (N) -- (K) -- (K) -- (L);
  \draw[thick,dashed] (L) -- (K); \draw[thick,dashed] (L) -- (M);  \draw[thick] (M) -- (N); \draw[thick] (N) -- (K);

  \coordinate(A) at (3,-0.9); \coordinate(B) at (4,0.4); \coordinate(C) at (4.3,3.8); \coordinate(D) at (2.9,3);
  \draw (A) -- (K); \draw (A) -- (B);  \draw[dashed] (B) -- (L); \draw (B) -- (C); \draw (M) -- (C); \draw (D) -- (C); \draw (D) -- (N); 
  \draw (D) -- (A);

  \coordinate(E) at (-3.5,-0.2); \coordinate(F) at (-1.6,0.8); \coordinate(G) at (-2,3.8); \coordinate(H) at (-3.3,2.9);
  \draw (E) -- (K); \draw (E) -- (F);  \draw[dashed] (F) -- (L); \draw[dashed] (F) -- (G); \draw (M) -- (G); \draw (H) -- (G); \draw (H) -- (N); 
  \draw (H) -- (E);

  \fill (K) circle (1pt); \fill (L) circle (1pt); \fill (M) circle (1pt); \fill (N) circle (1pt);
  \fill (A) circle (1pt); \fill (B) circle (1pt); \fill (C) circle (1pt); \fill (D) circle (1pt); \fill (E) circle (1pt); \fill (F) circle (1pt);
  \fill (G) circle (1pt); \fill (H) circle (1pt);

  \draw[->] (-0.05,0.05) -- (1,0.33);
  \draw[->] (0,-0.05) -- (1,-0.3);
  \draw[->] (-0.37,0.15) -- (-0.1,1.2);
  \draw[->] (-0.6,0.06) -- (-1.8,-0.015);
  \draw[->] (-0.57,0.15) -- (-0.3,1.2);
  \draw[->] (-0.25,0.15) -- (0.8,0.43);

  \node at (1.1,0.2) {\scriptsize $\xi_2$};
  \node at (0.1,1) {\scriptsize $\xi_3$};
  \node at (-0.6,1) {\scriptsize $\xi_3$};
  \node at (-1.6,0.2) {\scriptsize $\xi_2$};
  \node at (0.9,-0.1) {\scriptsize $\xi_1$};
   \node at (0.4,0.47) {\scriptsize $\xi_1$};
  \node at (0.75,1.8) {\scriptsize $\face^{(i)}$};
  \node[left] at (E) {\scriptsize $\vertex^{(i_8)}$}; %
  \node[left] at (F) {\scriptsize $\vertex^{(i_9)}$}; %
  \node[left] at (G) {\scriptsize $\vertex^{(i_{11})}$}; %
  \node[left] at (H) {\scriptsize $\vertex^{(i_{10})}$}; %
  \node[below] at (K) {\scriptsize $\vertex^{(i_0)}$}; %
  \node[above] at (L) {\; \; \; \scriptsize $\vertex^{(i_2)}$}; %
  \node[above] at (M) {\scriptsize $\vertex^{(i_6)}$}; %
  \node[above] at (N) {\scriptsize $\vertex^{(i_4)}$ \; }; %
  \node[right] at (A) {\scriptsize $\vertex^{(i_1)}$}; %
  \node[right] at (B) {\scriptsize $\vertex^{(i_3)}$}; %
  \node[right] at (C) {\scriptsize $\vertex^{(i_{7})}$}; %
  \node[right] at (D) {\scriptsize $\vertex^{(i_{5})}$}; %
  \node at (-1,1.8) {\scriptsize $\patch^{(i_1)}$};
  \node at (2.2,1.8) {\scriptsize$\patch^{(i_0)}$};

  \end{tikzpicture}
 }
 \end{center}
 \caption{For any inner face~$\face^{(i)}$, $i \in \indexFaceInner$, between the two neighboring patches~$\patch^{(i_0)}$ and $\patch^{(i_1)}$, $i_0,i_1 
 \in \indexPatch$, we assume that the common inner face~$\face^{(i)}$ is parameterized via the geometry mappings~$\map^{(i_0)}$ and 
 $\map^{(i_1)}$ as shown, and that the $12$ different vertices are labeled as visualized.}
 \label{fig:standard_two-patch_case}
\end{figure}

We define bivariate functions~$\alpha^{(i,i_0)}, \alpha^{(i,i_1)},\beta^{(i)}, \gamma^{(i)} :[0,1]^2 \rightarrow \R$ by
\begin{equation} \label{eq:gldata}
\begin{split}
\alpha^{(i,i_0)}(t_1,t_2) = \lambda \det\left(\partial_1 \map^{(i_0)}(0,t_1,t_2), \partial_2 \map^{(i_0)}(0,t_1,t_2), 
\partial_3 \map^{(i_0)}(0,t_1,t_2)  \right),  
\\
\alpha^{(i,i_1)}(t_1,t_2) = \lambda \det\left(\partial_1 \map^{(i_1)}(t_1,0,t_2), \partial_2 \map^{(i_1)}(t_1,0,t_2), 
\partial_3 \map^{(i_1)}(t_1,0,t_2)  \right), 
\\
\beta^{(i)}(t_1,t_2) = \lambda \det\left(\partial_2 \map^{(i_1)}(t_1,0,t_2), \partial_3 \map^{(i_1)}(t_1,0,t_2), 
\partial_1 \map^{(i_0)}(0,t_1,t_2)  \right),  
\\
\gamma^{(i)}(t_1,t_2) = \lambda \det\left(\partial_1 \map^{(i_1)}(t_1,0,t_2), \partial_2 \map^{(i_1)}(t_1,0,t_2), 
\partial_1 \map^{(i_0)}(0,t_1,t_2)  \right),  
\end{split}
\end{equation}
where $\lambda \in  (0,\infty) 
$ is given by minimizing the term
\[
||\alpha^{(i_0)} - 1  ||^2_{L_2} + || \alpha^{(i_1)} - 1 ||^2_{L_2},
\]
and which satisfy 
\begin{gather}
 \alpha^{(i,i_0)}(t_1,t_2) \partial_2 \map^{(i_1)}(t_1,0,t_2) + \alpha^{(i,i_1)}(t_1,t_2) \partial_1 \map^{(i_0)}(0,t_1,t_2) - \nonumber \\[-0.3cm] 
 \label{eq:cond_mapping} \\[-0.3cm]
 \beta^{(i)}(t_1,t_2) \partial_1 \map^{(i_1)}(t_1,0,t_2) - \gamma^{(i)}(t_1,t_2) \partial_3 \map^{(i_1)}(t_1,0,t_2)
 =  \f{0} \nonumber 
\end{gather}
for $(t_1,t_2) \in [0,1]^2$, cf.~\cite{BiKa19}. Since the geometry mappings~$\map^{(i_0)}$ and $\map^{(i_1)}$ are trilinear, non-singular 
parameterizations, the functions~$\alpha^{(i,i_0)}$, $\alpha^{(i,i_1)}$, $\beta^{(i)}$ and $\gamma^{(i)}$ are bivariate polynomials of 
bidegree~$(2,2)$, $(2,2)$, $(3,2)$ and $(2,3)$, respectively, and the functions~$\alpha^{(i,i_0)}$ and $\alpha^{(i,i_1)}$ further fulfill
\[
 \alpha^{(i,i_0)}(t_1,t_2) > 0 \mbox{ and } \alpha^{(i,i_1)}(t_1,t_2) >0,  \mbox{ } (t_1,t_2) \in [0,1]^2.
\]
It was also shown in~\cite{BiKa19}, that in case of a non-planar face~$\face^{(i)}$, there exist bilinear polynomial 
functions $\beta^{(i,i_0)}, \beta^{(i,i_1)}, \gamma^{(i,i_0)}, \gamma^{(i,i_1)}:[0,1]^2 \rightarrow \R$ given by 
\begin{equation*} \label{eq:split_beta}
\begin{split}
\beta^{(i,i_0)}(t_1,t_2)=\frac{\det (\vertex^{(i_6)}-\vertex^{(i_2)},\vertex^{(i_0)}-\vertex^{(i_4)},\partial_1 \map^{(i_0)}(0,t_1,t_2))}{\vol^{(i)}},\\
\beta^{(i,i_1)}(t_1,t_2)=\frac{\det (\vertex^{(i_6)}-\vertex^{(i_2)},\vertex^{(i_0)}-\vertex^{(i_4)},\partial_2 \map^{(i_1)}(t_1,0,t_2))}{\vol^{(i)}},\\
\end{split}
\end{equation*}
and
\begin{equation*} \label{eq:split_gamma}
\begin{split}
\gamma^{(i,i_0)}(t_1,t_2)=\frac{\det (\vertex^{(i_6)}-\vertex^{(i_4)},\vertex^{(i_2)}-\vertex^{(i_0)},\partial_1 \map^{(i_0)}(0,t_1,t_2))}{\vol^{(i)}},\\
\gamma^{(i,i_1)}(t_1,t_2)=\frac{\det (\vertex^{(i_6)}-\vertex^{(i_4)},\vertex^{(i_2)}-\vertex^{(i_0)},\partial_2 \map^{(i_1)}(t_1,0,t_2))}{\vol^{(i)}},\\
\end{split}
\end{equation*}
respectively, which satisfy
\begin{equation} \label{eq:beta}
  \beta^{(i)}(t_1,t_2) = \beta^{(i,i_0)}(t_1,t_2)\alpha^{(i,i_1)}(t_1,t_2) +  \beta^{(i,i_1)}(t_1,t_2)\alpha^{(i,i_0)}(t_1,t_2) 
\end{equation}
and
\begin{equation} \label{eq:gamma}
\gamma^{(i)}(t_1,t_2) = \gamma^{(i,i_0)}(t_1,t_2)\alpha^{(i,i_1)}(t_1,t_2) +  \gamma^{(i,i_1)}(t_1,t_2)\alpha^{(i,i_0)}(t_1,t_2), 
\end{equation}
where $\vol^{(i)}$ is the volume of the rectangular solid spanned by the three vectors~$\vertex^{(i_2)}-\vertex^{(i_0)}$, 
$\vertex^{(i_4)}-\vertex^{(i_0)}$ and $\vertex^{(i_6)}-\vertex^{(i_0)}$, i.e.
\[
 \vol^{(i)} = \det (\vertex^{(i_2)}-\vertex^{(i_0)},\vertex^{(i_4)}-\vertex^{(i_0)},\vertex^{(i_6)}-\vertex^{(i_0)}),
\] 
with the four vertices~$\vertex^{(i_0)}$, $\vertex^{(i_2)}$, $\vertex^{(i_4)}$ and $\vertex^{(i_6)}$ of the common 
face~$\face^{(i)}$. In addition, the functions~$\alpha^{(i,i_0)}$ and $\alpha^{(i,i_1)}$ can be written as
\begin{equation*} \label{eq:split_alpha0}
 \alpha^{(i,i_0)}(t_1,t_2) = \lambda \, \vol^{(i)} \left( \delta^{(i,i_0)}(t_1,t_2) -t_1 \gamma^{(i,i_0)}(t_1,t_2) - t_2 
 \beta^{(i,i_0)}(t_1,t_2) \right)
\end{equation*}
and
\begin{equation*} \label{eq:split_alpha1}
 \alpha^{(i,i_1)}(t_1,t_2) = - \lambda \, \vol^{(i)} \left( \delta^{(i,i_1)}(t_1,t_2) -t_1 \gamma^{(i,i_1)}(t_1,t_2) - t_2 
 \beta^{(i,i_1)}(t_1,t_2) \right),
\end{equation*}
with bilinear polynomial functions
\begin{equation*} \label{eq:delta0}
 \delta^{(i,i_0)}(t_1,t_2)=\frac{\det (\vertex^{(i_2)}-\vertex^{(i_0)},\vertex^{(i_4)}-\vertex^{(i_0)},\partial_1 
 \map^{(i_0)}(0,t_1,t_2))}{\vol^{(i)}}
\end{equation*}
and
\begin{equation*} \label{eq:delta1}
 \delta^{(i,i_1)}(t_1,t_2)=\frac{\det (\vertex^{(i_2)}-\vertex^{(i_0)},\vertex^{(i_4)}-\vertex^{(i_0)},\partial_2 \map^{(i_1)}(t_1,0,t_2))}{\vol^{(i)}}.
\end{equation*}

In case of a boundary face~$\face^{(i)}$, $i \in \indexFaceBoundary$, which is contained in the closure of the 
patch~$\patch^{(i_0)}$, $i_0 \in \indexPatch$, i.e. $\face^{(i)} \subset \overline{\patch^{(i_0)}}$, we will similarly assume as in the case of an 
inner face shown in Fig.~\ref{fig:standard_two-patch_case}, that the 
corresponding geometry mapping~$\map^{(i_0)}$ is (re)parameterized in such a way that the face~$\overline{\face^{(i)}}$ is given by
\begin{equation} \label{eq:standard_form_one_patch} 
 \map^{(i_0)}(0,t_1,t_2),  \mbox{ } (t_1,t_2) \in [0,1]^2, 
\end{equation}
and that the $8$ vertices of the patch~$\overline{\patch^{(i_0)}}$ are labeled as $\vertex^{(i_0)}, \vertex^{(i_1)}, \ldots, \vertex^{(i_7)}$. We say 
that the geometry mapping~$\map^{(i_0)}$ is given in \textit{standard form} with respect to its boundary face~$\face^{(i)}$ when it 
fulfills~\eqref{eq:standard_form_one_patch}.

Below, inspired by the work in~\cite{BiKa19} for the case of trilinearly parameterized two-patch volumes, we will assume that the 
considered trilinear multi-patch volumes~$\overline{\domain}$ satisfy the following assumption:

\begin{ass} \label{ass:domain}
Each inner face~$\face^{(i)}$, $i \in \indexFaceInner$, is non-planar, and its 
functions~$\alpha^{(i,i_0)}$, $\alpha^{(i,i_1)}$, 
$\beta^{(i)}$ and $\gamma^{(i)}$ possess full bidegrees~$(2,2)$, $(2,2)$, $(3,2)$ and $(2,3)$, respectively, its functions $\beta^{(i)}$ and 
$\gamma^{(i)}$ have no roots in the values~$(j_1 \, h, j_2 \, h)$, $j_1,j_2=1,2,\ldots,k$, and the greatest common divisor of its 
functions~$\alpha^{(i,i_0)}$ and $\alpha^{(i,i_1)}$ is a constant function.
\end{ass}
Note that each trilinear multi-patch volume~$\overline{\domain}$, which does not satisfy Assumption~\ref{ass:domain}, can be enforced to fulfill 
the assumption by slightly disturbing some of the values of their vertices~$\vertex^{(i)}$, $i \in \indexVertex$.

%%%%%%%%%%%%%%%%%%%%%%%%%%%%%%%%%%%%%%%%%%%%%%%%%%%%%%%%%%%%%%%%
\subsection{$C^1$ continuity condition across inner faces}
%%%%%%%%%%%%%%%%%%%%%%%%%%%%%%%%%%%%%%%%%%%%%%%%%%%%%%%%%%%%%%%% 

As already mentioned in Section~\ref{subsec:concept_C1spaces}, an isogeometric function~$\phi \in \V$ belongs to the $C^1$~isogeometric spline 
space~$\V^1$ if and only if for any two neighboring patches~$\patch^{(i_0)}$ and $\patch^{(i_1)}$, $i_0,i_1 \in \indexPatch$, with the common 
inner face~$\overline{\face^{(i)}} = \overline{\patch^{(i_0)}} \cap \overline{\patch^{(i_1)}}$, $i \in \indexFaceInner$, the associated graph patches 
$\left[\begin{array}{cc}
  \map^{(i_0)} &
  \phi \circ \map^{(i_0)}
 \end{array} \right]^T$
and 
$\left[\begin{array}{cc}
  \map^{(i_1)} & 
  \phi \circ \map^{(i_1)}
 \end{array} \right]^T$ 
meet with $G^1$ continuity. Let us consider two such neighboring patches $\patch^{(i_0)}$ and $\patch^{(i_1)}$, $i_0,i_1 \in \indexPatch$, 
and let us assume without loss of generality that the two corresponding geometry mappings~$\map^{(i_0)}$ and $\map^{(i_1)}$ are given in standard 
form~\eqref{eq:standard_form_two_patches} with respect to the common inner face~$\overline{\face^{(i)}}$, $i \in \indexFaceInner$. 
We further consider an isogeometric function~$\phi \in \V$, and recall that we denote 
by $\fstandard{i_0}$ and $\fstandard{i_1}$ its associated spline functions~$\phi \circ \map^{(i_0)}$ and $\phi \circ \map^{(i_1)}$, respectively. 
Then, the two graph patches $\left[\begin{array}{cc}
  \map^{(i_0)} &
  \fstandard{i_0}
 \end{array} \right]^T$
and 
$\left[\begin{array}{cc}
  \map^{(i_1)} & 
  \fstandard{i_1}
 \end{array} \right]^T$ 
are $G^1$-continuous if 
\[
 \left[\begin{array}{c}
  \map^{(i_0)} (0,t_1,t_2) \\
  \fstandard{i_0} (0,t_1,t_2)
 \end{array} \right]
 =
 \left[\begin{array}{c}
  \map^{(i_1)} (t_1,0,t_2) \\
  \fstandard{i_1} (t_1,0,t_2)
 \end{array} \right] , 
 %\mbox{ }
\quad 
 (t_1,t_2) \in [0,1]^2,
 \]
and if there exist bivariate functions~
$\widetilde{\alpha}^{(i,i_0)}, \widetilde{\alpha}^{(i,i_1)}, 
\widetilde{\beta}^{(i)}, \widetilde{\gamma}^{(i)}:[0,1]^2 \rightarrow \R$ with $\widetilde{\alpha}^{(i,i_0)} 
\widetilde{\alpha}^{(i,i_1)} >0$, such that
\begin{gather}
 \widetilde{\alpha}^{(i,i_0)}(t_1,t_2) \left[\begin{array}{c}
 \partial_2 \map^{(i_1)} (t_1,0,t_2) \\
 \partial_2 \fstandard{i_1} (t_1,0,t_2)
 \end{array} \right] + \widetilde{\alpha}^{(i,i_1)}(t_1,t_2) \left[\begin{array}{c}
  \partial_1 \map^{(i_0)} (0,t_1,t_2) \\
  \partial_1 \fstandard{i_0} (0,t_1,t_2)
 \end{array} \right] 
  - \nonumber \\[-0.3cm] 
 \label{eq:cond_all_G1} \\[-0.3cm] 
 \widetilde{\beta}^{(i)}(t_1,t_2) \left[\begin{array}{c}
  \partial_1 \map^{(i_1)} (t_1,0,t_2) \\
  \partial_1 \fstandard{i_1} (t_1,0,t_2)
 \end{array} \right] - \widetilde{\gamma}^{(i)}(t_1,t_2) \left[\begin{array}{c}
  \partial_3 \map^{(i_1)} (t_1,0,t_2) \\
  \partial_3 \fstandard{i_1} (t_1,0,t_2)
 \end{array} \right]
 =  \f{0}, \nonumber
\end{gather}
for $(t_1,t_2) \in [0,1]^2$. Since the geometry mappings~$\map^{(i_0)}$ and $\map^{(i_1)}$ are given, the first three coordinate rows in 
equation~\eqref{eq:cond_all_G1} uniquely determine the functions~$\widetilde{\alpha}^{(i,i_0)}$, $\widetilde{\alpha}^{(i,i_1)}$, 
$\widetilde{\beta}^{(i)}$ and $\widetilde{\gamma}^{(i)}$ up to a common function~$\widetilde{\lambda}:[0,1]^2 \rightarrow \R$. Without loss of 
generality, we choose the functions~$\widetilde{\alpha}^{(i,i_0)}$, $\widetilde{\alpha}^{(i,i_1)}$, $\widetilde{\beta}^{(i)}$ and 
$\widetilde{\gamma}^{(i)}$ as the functions $\alpha^{(i,i_0)}$, $\alpha^{(i,i_1)}$, $\beta^{(i)}$ and $\gamma^{(i)}$ defined in 
equation~\eqref{eq:gldata}, which fulfill the first three coordinate rows in equation~\eqref{eq:cond_all_G1} , cf. equation~\eqref{eq:cond_mapping}. 
Therefore, the isogeometric function~$\phi$ is $C^1$ across the inner face~$\overline{\face^{(i)}}$ if and only if
\begin{equation} \label{eq:cond_C0}
 \fstandard{i_0}(0,t_1,t_2) = \fstandard{i_1}(t_1,0,t_2),  \mbox{ } (t_1,t_2) \in [0,1]^2, 
\end{equation} 
and
\begin{gather} 
 \alpha^{(i,i_0)}(t_1,t_2) \partial_2 \fstandard{i_1}(t_1,0,t_2) + \alpha^{(i,i_1)}(t_1,t_2) \partial_1 \fstandard{i_0}(0,t_1,t_2) - \nonumber \\[-0.3cm] \label{eq:cond_C1_1} \\[-0.3cm] 
 \beta^{(i)}(t_1,t_2) \partial_1 \fstandard{i_1}(t_1,0,t_2) - \gamma^{(i)}(t_1,t_2) \partial_3 \fstandard{i_1}(t_1,0,t_2)
 =  0, \nonumber
\end{gather}
for $(t_1,t_2) \in [0,1]^2$, and the function~$\phi$ is globally $C^1$ on~$\overline{\domain}$, i.e. $\phi \in \V^1$, if and only if the pull-back of the
function~$\phi$ satisfies equations~\eqref{eq:cond_C0} and \eqref{eq:cond_C1_1} for any inner face~$\overline{\face^{(j)}}$, $j \in \indexFaceInner$.

Let us consider the $C^1$ continuity condition~\eqref{eq:cond_C1_1} of the isogeometric function~$\phi$ across the inner 
face~$\overline{\face^{(i)}}$ in more detail. By using relations~\eqref{eq:beta}, \eqref{eq:gamma} and \eqref{eq:cond_C0}, the 
condition~\eqref{eq:cond_C1_1} is equivalent to
\begin{gather} 
\frac{\partial_1 \fstandard{i_0}(0,t_1,t_2) - \beta^{(i,i_0)}(t_1,t_2)\partial_2 \fstandard{i_0}(0,t_1,t_2) - \gamma^{(i,i_0)}(t_1,t_2) 
\partial_3 \fstandard{i_0}(0,t_1,t_2) }{\alpha^{(i,i_0)}(t_1,t_2)} = \nonumber \\[-0.3cm] 
 \label{eq:cond_C1_2} \\[-0.3cm] 
 - \frac{\partial_2 \fstandard{i_1}(t_1,0,t_2) - \beta^{(i,i_1)}(t_1,t_2)\partial_1 \fstandard{i_1}(t_1,0,t_2) - \gamma^{(i,i_1)}(t_1,t_2) 
\partial_3 \fstandard{i_1}(t_1,0,t_2) }{\alpha^{(i,i_1)}(t_1,t_2)} . \nonumber
\end{gather}
Let $\ftrace :[0,1]^2 \rightarrow \R$ and $\fder: [0,1]^2 \rightarrow \R$ be the two bivariate functions, which are determined by the equally valued 
terms in equation~\eqref{eq:cond_C0} and \eqref{eq:cond_C1_2}, respectively. While the function~$\ftrace$ is just the trace of the function~$\phi$ at 
the common inner face~$\overline{\Gamma^{(i)}}$, i.e. 
\[
\ftrace(t_1,t_2)=\phi \circ \map^{(i_0)}(0,t_1,t_2) = \phi \circ \map^{(i_1)}(t_1,0,t_2), 
\quad (t_1,t_2) \in [0,1]^2,
\]
the function~$\fder$ is the directional derivative of $\phi$ with respect to the transversal direction~$\f{d}^{(i)}$ at the common inner
face~$\overline{\face^{(i)}}$, i.e.
\[
 \fder(t_1,t_2)=\nabla \phi \cdot \f{d}^{(i)} \circ \map^{(i_0)}(0,t_1,t_2) = \nabla \phi \cdot \f{d}^{(i)} \circ \map^{(i_1)}(t_1,0,t_2), 
 %\mbox{ }
 \quad (t_1,t_2) \in [0,1]^2,
\]
with $\f{d}^{(i)}=\f{d}^{(i,i_0)}=\f{d}^{(i,i_1)}$ on~$\overline{\face^{(i)}}$ given as
\[
 \f{d}^{(i,i_0)} \circ \map^{(i_0)}(0,t_1,t_2)= \grad \f{F}^{(i_0)}(0,t_1,t_2) \cdot  \left[\begin{array}{ccc} -\beta^{(i,i_0)}(t_1,t_2) &   -\gamma^{(i,i_0)}(t_1,t_2) &  1
 \hspace{-0.15cm}  \end{array} \right]^T  
 \frac{1}{\alpha^{(i,i_0)}(t_1,t_2)}
\]
and
\[
 \f{d}^{(i,i_1)} \circ \map^{(i_1)}(t_1,0,t_2)= \grad \f{F}^{(i_1)}(t_1,0,t_2) \cdot  \left[\begin{array}{ccc}
   \beta^{(i,i_1)}(t_1,t_2) &   \gamma^{(i,i_1)}(t_1,t_2) &   -1  
 \end{array} \right]^T \frac{1}{\alpha^{(i,i_1)}(t_1,t_2)},
\]
cf.~\cite{BiKa19}. The use of the functions~$\ftrace$ and $\fder$ directly implies that
\[
 \fstandard{i_0}(0,t_1,t_2) = \fstandard{i_1}(t_1,0,t_2) = \ftrace(t_1,t_2),
\]
\[
 \partial_1 \fstandard{i_0}(0,t_1,t_2) = \beta^{(i,i_0)}(t_1,t_2) \partial_1 \ftrace(t_1,t_2) + \gamma^{(i,i_0)}(t_1,t_2) \partial_2 \ftrace(t_1,t_2) 
 + \alpha^{(i,i_0)}(t_1,t_2) \fder(t_1,t_2),
\]
and
\[
 \partial_2 \fstandard{i_1}(t_1,0,t_2) = \beta^{(i,i_1)}(t_1,t_2) \partial_1 \ftrace(t_1,t_2) + \gamma^{(i,i_1)}(t_1,t_2) \partial_2 \ftrace(t_1,t_2) 
 - \alpha^{(i,i_1)}(t_1,t_2) \fder(t_1,t_2).
\] 
Then, the Taylor expansion of $\fstandard{i_0} = \phi \circ \map^{(i_0)}$ at $(\xi_1,\xi_2,\xi_3) = (0,\xi_2,\xi_3)$, and the 
Taylor expansion of $\fstandard{i_1} = \phi \circ \map^{(i_1)}$ at $(\xi_1,\xi_2,\xi_3) = (\xi_1,0,\xi_3)$ is given by
\begin{eqnarray}
\fstandard{i_0}(\xi_1,\xi_2,\xi_3) & = & \ftrace(\xi_2,\xi_3) +  \big( \beta^{(i,i_0)}(\xi_2,\xi_3) \partial_1 \ftrace(\xi_2,\xi_3) + \nonumber \\[-0.3cm] 
 \label{eq:Taylor_f0} \\[-0.3cm]
 & &  \gamma^{(i,i_0)}(\xi_2,\xi_3) \partial_2 \ftrace(\xi_2,\xi_3) + \alpha^{(i,i_0)}(\xi_2,\xi_3) \fder(\xi_2,\xi_3) \big) \xi_1 + 
  \mathcal{O}(\xi_1^2), \nonumber
\end{eqnarray}
and by 
\begin{eqnarray}  
\fstandard{i_1}(\xi_1,\xi_2,\xi_3) & = & \ftrace(\xi_1,\xi_3) +  \big( \beta^{(i,i_1)}(\xi_1,\xi_3) \partial_1 \ftrace(\xi_1,\xi_3) + \nonumber \\[-0.3cm] 
 \label{eq:Taylor_f1} \\[-0.3cm]
  & & \gamma^{(i,i_1)}(\xi_1,\xi_3) \partial_2 \ftrace(\xi_1,\xi_3) - \alpha^{(i,i_1)}(\xi_1, \xi_3) \fder(\xi_1,\xi_3) \big) \xi_2 + 
  \mathcal{O}(\xi_2^2) , \nonumber
\end{eqnarray}
respectively, cf.~\cite{BiKa19}. The representations~\eqref{eq:Taylor_f0} and \eqref{eq:Taylor_f1} of an isogeometric function~$\phi \in \V^1$ 
along an inner face~$\overline{\face^{(i)}}$, $i \in \indexFaceInner$, will be used in the next sections to study and to 
design the $C^1$~isogeometric spline space~$\V^1$ for trilinear multi-patch volumes~$\overline{\domain}$.

%%%%%%%%%%%%%%%%%%%%%%%%%%%%%%%%%%%%%%%%%%%%%%%%%%%%%%%%%%%%%%%%%%%%%%%%%%%%%%%%%%%%%%%%%%%%%%%%%%
\section{Design and study of $C^1$ isogeometric spline spaces} \label{sec:design_study_C1spaces}
%%%%%%%%%%%%%%%%%%%%%%%%%%%%%%%%%%%%%%%%%%%%%%%%%%%%%%%%%%%%%%%%%%%%%%%%%%%%%%%%%%%%%%%%%%%%%%%%%

In~\cite{BiKa19}, the $C^1$ isogeometric spline space~$\V^1$ for the case of a trilinear two-patch volume~$\overline{\domain}$ was studied, and a 
basis for the space~$\V^1$ was constructed. This basis consists of $C^1$~functions, which possess simple explicit representations and have small local 
supports, and will be used in this and in the next section as a tool for the investigation of the $C^1$~isogeometric spline space~$\V^1$ for 
the case of a trilinear multi-patch volume~$\overline{\domain}$. More precisely, we will present in this section for the trilinear multi-patch 
volumes~$\overline{\domain}$ a general framework for design of the corresponding $C^1$ isogeometric spline space~$\V^1$ and of an associated basis, and will 
then focus in the next section on a more detailed study for a particular subclass of trilinearly parameterized multi-patch volumes~$\overline{\domain}$.
Before, we will introduce some needed additional notations and definitions, and will briefly recall the construction~\cite{BiKa19} for the two-patch case.

%%%%%%%%%%%%%%%%%%%%%%%%%%%%%%%%%%%%%%%%%%%%%%%%%%%%%%%%
\subsection{Preliminaries} \label{subsec:preliminaries}
%%%%%%%%%%%%%%%%%%%%%%%%%%%%%%%%%%%%%%%%%%%%%%%%%%%%%%%%

Let $\I_0=\{0,1,\ldots,n_{0}-1 \}$ and $\I_1=\{0,1,\ldots, n_{1}-1 \}$ with $n_0=\dim \US{p}{r+1}{h} =  p +1 + k(p-r-1)$ and 
$n_1=\dim \US{p-2}{r}{h} =  p-1 + k(p-r-2)$ be the index sets of the univariate 
B-splines~$\UN{p}{r+1}{j}$ and $\UN{p-2}{r}{j}$ of the spline spaces~$\US{p}{r+1}{h}$ and $\US{p-2}{r}{h}$, respectively. In addition, let 
$\UR{p}{r+1}{j}: [0,1] \rightarrow \R$, $j \in \I_0$, be the spline functions defined by
\[
 \UR{p}{r+1}{j}(\xi)= \begin{cases}
                           -h \UN{p-1}{r}{j}(\xi) & \mbox{if } j =0, \\
                           h\,(\UN{p-1}{r}{j-1}(\xi) - \UN{p-1}{r}{j}(\xi)  ) \; & \mbox{if } j \in \I_0 \setminus \{ 0,n_0-1 \},\\
                           h \UN{p-1}{r}{j-1}(\xi) \; & \mbox{if } j=n_0-1,
                          \end{cases}
\]
and let $\UR{\f{p}}{\f{r+1}}{j_1,j_2}$ be the tensor-product spline functions $\UR{\f{p}}{\f{r+1}}{j_1,j_2} = 
\UR{p}{r+1}{j_1}\UR{p}{r+1}{j_2}$, $j_1,j_2 \in \I_0$.
We further denote by $\UM{p}{r}{j}:[0,1] \rightarrow \R$, $j=0,1$, the basis transformation of the B-splines~$\UN{p}{r}{j}$, $j=0,1$, given by
\[
\UM{p}{r}{0}(\xi) = \UN{p}{r}{0}(\xi) + \UN{p}{r}{1}(\xi) \quad {\rm and } \quad \UM{p}{r}{1}(\xi) = \frac{h}{p} \UN{p}{r}{1}(\xi),
\]
which fulfill $\partial_{\xi}^{i} \UM{p}{r}{j}(0) = \delta_j^i$, $i=0,1$, where $\delta_j^i$ is the Kronecker delta.

Recall that for any inner face~$\face^{(i)}$, $i \in \indexFaceInner$, with the two neighboring patches~$\patch^{(i_0)}$ and 
$\patch^{(i_1)}$, $i_0,i_1 \in \indexPatch$, i.e. $\overline{\face^{(i)}} = \overline{\patch^{(i_0)}} \cap \overline{\patch^{(i_1)}} $, we can assume
that the two associated geometry mappings~$\map^{(i_0)}$ and $\map^{(i_1)}$ can be reparameterized (if necessary) to be in standard 
form~\eqref{eq:standard_form_two_patches}. Similarly, for any boundary face~$\face^{(i)}$, $i \in \indexFaceBoundary$, which is contained in the 
closure of the patch~$\patch^{(i_0)}$, $i_0 \in \indexPatch$, i.e. $\face^{(i)} \subset \overline{\patch^{(i_0)}}$, the associated geometry 
mapping~$\map^{(i_0)}$ can be assumed to be in standard form~\eqref{eq:standard_form_one_patch}. 

We define on the trilinear multi-patch volume~$\overline{\domain}$ for each patch~$\patch^{(i)}$, $i \in \indexPatch$, the isogeometric 
functions~$\phiPatch{i}{j_1}{j_2}{j_3}$, $i=0,1$, $j_1,j_2,j_3 \in \I$, for each boundary face~$\face^{(i)}$, $i \in \indexFaceBoundary$, the 
isogeometric functions $\phiFace{i}{j_1}{j_2}{j_3}$, $j_1=0,1$, $j_2,j_3 \in \I$,
and for each inner face~$\face^{(i)}$, $i \in \indexFaceInner$, the isogeometric functions $\phiFace{i}{j_1}{j_2}{j_3}$, $j_1=0,1$, $j_2,j_3 \in \I_{j_1}$, 
whose spline functions~$\fPatch{i}{j_1}{j_2}{j_3}{\tau} = \phiPatch{i}{j_1}{j_2}{j_3} \circ \map^{(\tau)}$ and $\fFace{i}{j_1}{j_2}{j_3}{\tau} = 
\phiFace{i}{j_1}{j_2}{j_3} \circ \map^{(\tau)}$ for $\tau \in \indexPatch$, possess the form
\begin{equation*}
 \fPatch{i}{j_1}{j_2}{j_3}{\tau}(\xi_1,\xi_2,\xi_3) = 
 \begin{cases}
  \UN{\f{p}}{\f{r}}{j_1,j_2,j_3}(\xi_1,\xi_2,\xi_3) & \mbox{if }\tau=i, \\
  0 & \mbox{otherwise},
 \end{cases}
\end{equation*}
\begin{equation*}
 \fFace{i}{j_1}{j_2}{j_3}{\tau}(\xi_1,\xi_2,\xi_3) =  \fPatch{i_0}{j_1}{j_2}{j_3}{\tau}(\xi_1,\xi_2,\xi_3),
\end{equation*}
for $i \in \indexFaceBoundary$, and
\begin{equation*} \label{eq:edge0}
 \fFace{i}{0}{j_2}{j_3}{\tau}(\xi_1,\xi_2,\xi_3) = 
 \begin{cases}
 \begin{array}{c}
  \UN{\f{p}}{\f{r+1}}{j_2,j_3}(\xi_2,\xi_3) \UM{p}{r}{0}(\xi_1)+  \big( \beta^{(i,i_0)}(\xi_2,\xi_3) \partial_1 
  \UN{\f{p}}{\f{r+1}}{j_2,j_3}(\xi_2,\xi_3) + \\
  \gamma^{(i,i_0)}(\xi_2,\xi_3) \partial_2 \UN{\f{p}}{\f{r+1}}{j_2,j_3}(\xi_2,\xi_3) + \\
  \alpha^{(i,i_0)}(\xi_2,\xi_3) \frac{p}{\lambda \, \vol^{(i)}} \UR{\f{p}}{\f{r+1}}{j_2,j_3}(\xi_2,\xi_3) \big) \UM{p}{r}{1}(\xi_1),
 \end{array}
 & $\hspace{-0.4cm} $\mbox{if }\tau=i_0, \\
  \begin{array}{c}
  \UN{\f{p}}{\f{r+1}}{j_2,j_3}(\xi_1,\xi_3) \UM{p}{r}{0}(\xi_2)+  \big( \beta^{(i,i_1)}(\xi_1,\xi_3) \partial_1 
  \UN{\f{p}}{\f{r+1}}{j_2,j_3}(\xi_1,\xi_3) +  \\
  \gamma^{(i,i_1)}(\xi_1,\xi_3) \partial_2 \UN{\f{p}}{\f{r+1}}{j_2,j_3}(\xi_1,\xi_3) - \\
 \alpha^{(i,i_1)}(\xi_1,\xi_3) \frac{p}{\lambda \, \vol^{(i)}} \UR{\f{p}}{\f{r+1}}{j_2,j_3}(\xi_1,\xi_3) \big) \UM{p}{r}{1}(\xi_2),
 \end{array}
  & $\hspace{-0.4cm} $ \mbox{if }\tau=i_1, \\
  \quad 0 & $\hspace{-0.6cm} $ \mbox{otherwise},
 \end{cases}
 \end{equation*}
\begin{equation*} \label{eq:edge1}
 \fFace{i}{1}{j_2}{j_3}{\tau}(\xi_1,\xi_2,\xi_3) = 
 \begin{cases}
  \alpha^{(i,i_0)}(\xi_2,\xi_3) \UN{\f{p-2}}{\f{r}}{j_2,j_3}(\xi_2,\xi_3) \UM{p}{r}{1}(\xi_1)  & \mbox{if }\tau=i_0, \\
  -\alpha^{(i,i_1)}(\xi_1,\xi_3) \UN{\f{p-2}}{\f{r}}{j_2,j_3}(\xi_1,\xi_3) \UM{p}{r}{1}(\xi_2)  & \mbox{if } \tau=i_1, \\
  0 & \mbox{otherwise},
 \end{cases}
\end{equation*}
for $i \in \indexFaceInner$.

The functions $\phiPatch{i}{j_1}{j_2}{j_3}$, $i \in \indexPatch$, $j_1,j_2,j_3 \in \I$, are just the standard isogeometric spline functions, 
and are $C^1$ on $\overline{\domain}$, if they have vanishing values and gradients at 
all inner faces~$\overline{\face^{(\ell)}}$, $\ell \in \indexFaceInner$, i.e. $\phiPatch{i}{j_1}{j_2}{j_3}(\overline{\face^{(\ell)}})=0$ and 
$\nabla \phiPatch{i}{j_1}{j_2}{j_3}(\overline{\face^{(\ell)}})=\f{0}$. For $j_1,j_2,j_3 \in \I \setminus \{0,1,n-2,n-1 \}$, it is guaranteed that 
the function $\phiPatch{i}{j_1}{j_2}{j_3}$ belongs to the space~$\V^1$.

For a boundary face~$\face^{(i)}$, $i \in \indexFaceBoundary$, the functions $\phiFace{i}{j_1}{j_2}{j_3}$, $j_1=0,1$, $j_2,j_3 \in \I$, are 
again just the standard isogeometric spline functions, which are analogously as before $C^1$ on $\overline{\domain}$, if the functions possess 
vanishing values and gradients at all inner faces~$\overline{\face^{(\ell)}}$, $\ell \in \indexFaceInner$, i.e. 
$\phiFace{i}{j_1}{j_2}{j_3}(\overline{\face^{(\ell)}})=0$ and $\nabla \phiFace{i}{j_1}{j_2}{j_3}(\overline{\face^{(\ell)}})=\f{0}$.
Now, the function $\phiFace{i}{j_1}{j_2}{j_3}$ belongs in any case to the space~$\V^1$, if $j_1=0,1$, $j_2,j_3 \in \I \setminus \{0,1,n-2,n-1 \}$.

For an inner face~$\face^{(i)}$, $i \in \indexFaceInner$, the functions $\phiFace{i}{j_1}{j_2}{j_3}$, $j_1=0,1$, $j_2,j_3 \in \I_{j_1}$, 
are defined by the spline representations~\eqref{eq:Taylor_f0} and \eqref{eq:Taylor_f1}, which are restricted in $\xi_1$- and $\xi_2$-direction, 
respectively, to the first two B-splines~$\UN{p}{r}{0}$ and $\UN{p}{r}{1}$, and are obtained by selecting the functions~$\ftrace$ and $\fder$ as 
\begin{equation} \label{eq:f0f1_func0}
\ftrace(t_1,t_2)=\UN{\f{p}}{\f{r+1}}{j_2,j_3}(t_1,t_2) \mbox{ and }\fder(t_1,t_2)=\frac{p}{\lambda \, \vol^{(i)}} \UR{\f{p}}{\f{r+1}}{j_2,j_3}(t_1,t_2)  
\end{equation}
for the functions~$\phiFace{i}{0}{j_2}{j_3}$ and as 
\begin{equation} \label{eq:f0f1_func1}
\ftrace(t_1,t_2)=0 \mbox{ and } \fder(t_1,t_2)=\UN{\f{p-2}}{\f{r}}{j_2,j_3}(t_1,t_2)
\end{equation}
for the functions~$\phiFace{i}{1}{j_2}{j_3}$. The use of the representations~\eqref{eq:Taylor_f0} and
\eqref{eq:Taylor_f1} implies that the functions $\phiFace{i}{j_1}{j_2}{j_3}$, $i \in \indexFaceInner$, $j_1=0,1$, $j_2,j_3 \in \I_{j_1}$, are $C^1$ 
on the two-patch volume $\overline{\patch^{(i_0)}} \cup \overline{\patch^{(i_1)}}$, and the choices~\eqref{eq:f0f1_func0} and \eqref{eq:f0f1_func1} 
for the functions~$\ftrace$ and $\fder$ further guarantee that the functions~$\phiFace{i}{j_1}{j_2}{j_3}$ are linearly independent and that their 
spline functions $\fFace{i}{j_1}{j_2}{j_3}{\tau}$ belong to the space $\US{\f{p}}{\f{r}}{h}$ for all $\tau \in \indexPatch$, cf.~\cite{BiKa19}. If a  
function $\fFace{i}{j_1}{j_2}{j_3}{ \tau}$ has vanishing values and gradients at all other inner faces~$\overline{\face^{(\ell)}}$, $\ell \in 
\indexFaceInner \setminus \{ i \}$, i.e. $\phiFace{i}{j_1}{j_2}{j_3}(\overline{\face^{(\ell)}})=0$ and 
$\nabla \phiFace{i}{j_1}{j_2}{j_3}(\overline{\face^{(\ell)}})=\f{0}$, the function is also $C^1$ on $\overline{\domain}$, and therefore belongs to 
the space~$\V^1$. This is guaranteed for the case of $j_1=0,1$, $j_2,j_3 \in \I_{j_1} \setminus \{0,\ldots,2-j_1,n_{j_1}-3+j_1,\ldots,n_{j_1}-1 \}$.  

%%%%%%%%%%%%%%%%%%%%%%%%%%%%%%%%%%%%%%%%%%%%%%%%%%%%%%%%%%%%%%%%%%%
\subsection{The two-patch case} \label{subsec:two_patch_case}
%%%%%%%%%%%%%%%%%%%%%%%%%%%%%%%%%%%%%%%%%%%%%%%%%%%%%%%%%%%%%%%%%%%

In this subsection, let us restrict to a trilinear two-patch volume~$\overline{\domain} = \overline{\patch^{(0)}} 
\cup \overline{\patch^{(1)}}$, where the one inner face $\overline{\patch^{(0)}} \cap \overline{\patch^{(1)}}$ is labeled by $\overline{\face^{(0)}}$.
We assume without loss of generality that the two geometry mappings~$\map^{(0)}$ and $\map^{(1)}$ are given in standard 
form~\eqref{eq:standard_form_two_patches}, that is, the face~$\overline{\face^{(0)}}$ is parameterized by
\[
\map^{(0)}(0,t_1,t_2) = \map^{(1)}(t_1,0,t_2),  \mbox{ } (t_1,t_2) \in [0,1]^2.
\]
Let us recall now the construction~\cite{BiKa19} for the $C^1$ isogeometric spline space~$\V^1$ over the two-patch volume~$\overline{\domain}$.
The $C^1$ isogeometric spline space~$\V^1$ can be decomposed into the direct sum
\begin{equation} \label{eq:direct_sum_TwoPatch}
 \V^1 = \SpacePatchTwo{0} \oplus \SpacePatchTwo{1} \oplus \SpaceFaceTwo{0},
\end{equation}
with the subspaces
\begin{eqnarray*}
 \SpacePatchTwo{0} &  = &  \Big\{ \phi \in \V^1 \; | \; \fstandard{1}(\xi_1,\xi_2,\xi_3)=0, \; \fstandard{0}(\xi_1,\xi_2,\xi_3)= \\
 & &  \sum_{j_1=2}^{n-1} \sum_{j_2=0}^{n-1} \sum_{j_3=0}^{n-1} \coeff_{j_1,j_2,j_3}^{(0)} \UN{\f{p}}{\f{r}}{j_1,j_2,j_3} (\xi_1,\xi_2,\xi_3), \;
 \coeff_{j_1,j_2,j_3}^{(0)} \in \R \Big\} ,  
\end{eqnarray*}
\begin{eqnarray*}
 \SpacePatchTwo{1} &  = & \Big\{ \phi \in \V^1 \; | \; \fstandard{0}(\xi_1,\xi_2,\xi_3)=0, \; \fstandard{1}(\xi_1,\xi_2,\xi_3)= \\
 & & \sum_{j_1=0}^{n-1} \sum_{j_2=2}^{n-1} \sum_{j_3=0}^{n-1} \coeff_{j_1,j_2,j_3}^{(1)} \UN{\f{p}}{\f{r}}{j_1,j_2,j_3} (\xi_1,\xi_2,\xi_3), \;
 \coeff_{j_1,j_2,j_3}^{(1)} \in \R \Big\} , 
\end{eqnarray*}
and
\begin{eqnarray}
 \SpaceFaceTwo{0} & = & \Big\{ \phi \in \V^1 \; | \; \fstandard{0}(\xi_1,\xi_2,\xi_3)=\sum_{j_1=0}^{1} \sum_{j_2=0}^{n-1} \sum_{j_3=0}^{n-1} 
 \coeff_{j_1,j_2,j_3}^{(0)} \UN{\f{p}}{\f{r}}{j_1,j_2,j_3} (\xi_1,\xi_2,\xi_3),
 \nonumber \\
 & & \fstandard{1}(\xi_1,\xi_2,\xi_3) = \sum_{j_1=0}^{n-1} \sum_{j_2=0}^{1} \sum_{j_3=0}^{n-1} \coeff_{j_1,j_2,j_3}^{(1)} 
 \UN{\f{p}}{\f{r}}{j_1,j_2,j_3} (\xi_1,\xi_2,\xi_3), \; \coeff_{j_1,j_2,j_3}^{(i)} \in \R, \; i=0,1 \Big\} . \nonumber 
\end{eqnarray}
The subspaces~$\SpacePatchTwo{0}$, $\SpacePatchTwo{1}$ and $\SpaceFaceTwo{0}$ can be equivalently described as
\begin{equation*}
 \SpacePatchTwo{0} = \Span \left\{ \phiPatch{0}{j_1}{j_2}{j_3}  \;|\; j_2,j_3 \in \I, \, j_1 \in \I \setminus \{0,1 \}  \right\},
\end{equation*}
\begin{equation*}
 \SpacePatchTwo{1} = \Span \left\{ \phiPatch{1}{j_1}{j_2}{j_3}  \;|\; j_1,j_3 \in \I,\, j_2 \in \I \setminus \{0,1 \}  \right\},
\end{equation*}
and
\begin{equation*} 
 \SpaceFaceTwo{0} = \Span \left\{ \phiFace{0}{j_1}{j_2}{j_3}  \;|\; j_1=0,1, \; j_2,j_3 \in \I_{j_1}  \right\},
\end{equation*} 
respectively, cf.~\cite{BiKa19}. Then, as a direct consequence of the use of the direct sum~\eqref{eq:direct_sum_TwoPatch} for the 
representation of the $C^1$ space~$\V^1$, the dimension of $\V^1$ is given by
\[
\dim \V^1 = \dim \SpacePatchTwo{0} + \dim \SpacePatchTwo{1} + \dim \SpaceFaceTwo{0} =  2 |\I|^2(|\I|-2) + |\I_0|^2 + |\I_1|^2.
\]

%%%%%%%%%%%%%%%%%%%%%%%%%%%%%%%%%%%%%%%%%%%%%%%%%%%%%%%%%%%%%%%%%%%%%%%%%%%%%%%%%%%%%%%%
\subsection{A general framework for the design} \label{subsec:general_framework}
%%%%%%%%%%%%%%%%%%%%%%%%%%%%%%%%%%%%%%%%%%%%%%%%%%%%%%%%%%%%%%%%%%%%%%%%%%%%%%%%%%%%%%%%

We will present a general framework for the construction of the $C^1$ isogeometric spline space~$\V^1$ over a trilinear multi-patch 
volume~$\overline{\domain}$, which will be based as in the two-patch case on the decomposition of the space~$\V^1$ into the direct sum 
of simpler subspaces. 

Clearly, the space~$\V^1$ can be just described as the direct sum
\begin{equation*}
 \V^1= \left( \bigoplus_{i \in \indexPatch} \SpacePatch{i} \right) \oplus \SpaceFaceAll,
\end{equation*}
where the single subspaces~$\SpacePatch{i}$, $i \in \indexPatch$, and $\SpaceFaceAll$ are given as
\begin{eqnarray*}
 \SpacePatch{i}  & = &  \Big\{ \phi \in \V^1 \; | \; \fstandard{i}(\xi_1,\xi_2,\xi_3) = \sum_{j_1=2}^{n-3} \sum_{j_2=2}^{n-3} \sum_{j_3=2}^{n-3} \coeff_{j_1,j_2,j_3}^{(i)}
\UN{\f{p}}{\f{r}}{j_1,j_2,j_3}(\xi_1,\xi_2,\xi_3), \; \coeff_{j_1,j_2,j_3}^{(i)} \in \R, \\
& & \fstandard{\ell}(\xi_1,\xi_2,\xi_3) = 0,\, \ell \in \indexPatch \setminus \{i\} \Big\}, 
\end{eqnarray*}
and
\begin{eqnarray}
 \SpaceFaceAll & = & \Big\{ \phi \in \V^1 \; | \; \fstandard{i}(\xi_1,\xi_2,\xi_3) = \sum_{j_1=0}^{n-1} \sum_{j_2=0}^{n-1} \sum_{j_3=0}^{n-1} 
 \coeff_{j_1,j_2,j_3}^{(i)} \UN{\f{p}}{\f{r}}{j_1,j_2,j_3}(\xi_1,\xi_2,\xi_3), 
 \nonumber \\ 
 & & \coeff_{j_1,j_2,j_3}^{(i)}\in \R \mbox{ for } (j_1,j_2,j_3) \in \I^3 \setminus (\I \setminus \{0,1,n-2,n-1 \})^3, \nonumber \\
 & & \coeff_{j_1,j_2,j_3}^{(i)} =0 \mbox{ for } (j_1,j_2,j_3) \in (\I \setminus \{0,1,n-2,n-1 \})^3, \; i \in \indexPatch \Big\}, \nonumber 
\end{eqnarray}
respectively. Note that the subspaces~$\SpacePatch{i}$, $i \in \indexPatch$, are simply equal to
\begin{equation} \label{eq:patch_space}
 \SpacePatch{i} = \Span \left\{ \phiPatch{i}{j_1}{j_2}{j_3}  \;|\; j_1,j_2,j_3 \in \I \setminus \{0,1,n-2,n-1 \} \right\}.
\end{equation}
As already mentioned in Section~\ref{subsec:preliminaries}, the functions~$\phiPatch{i}{j_1}{j_2}{j_3}$, $ j_1,j_2,j_3 
\in \I \setminus \{0,1,n-2,n-1 \}$, are trivially $C^1$ on $\overline{\domain}$ and therefore belong to the space~$\V^1$, since they have vanishing 
values and gradients at all faces~$\overline{\Gamma^{(\ell)}}$, $\ell \in \indexFace$, i.e. $\phiPatch{i}{j_1}{j_2}{j_3}(\overline{\face^{(\ell)}})=0$ 
and $\nabla \phiPatch{i}{j_1}{j_2}{j_3}(\overline{\face^{(\ell)}})=\f{0}$.

To study the space~$\SpaceFaceAll$ in more detail, we need some additional notations and definitions. Let $\phiFaceAll{i}:\overline{\domain} 
\rightarrow \R$, $i \in \indexFace$, be the isogeometric function defined as the linear combination of all functions~$\phiFace{i}{j_1}{j_2}{j_3}$, i.e.
\[
 \phiFaceAll{i}(\f{x})= \sum_{j_1=0}^{1} \sum_{j_2=0}^{n_{j_1}-1} \sum_{j_3=0}^{n_{j_1}-1} \coeffFace{i}{j_1}{j_2}{j_3} 
 \phiFace{i}{j_1}{j_2}{j_3}(\f{x}), \quad \coeffFace{i}{j_1}{j_2}{j_3} \in \R,
\]
for the case of an inner face~$\face^{(i)}$, $i \in \indexFaceInner$, and
\[
 \phiFaceAll{i}(\f{x})= \sum_{j_1=0}^{1} \sum_{j_2=0}^{n-1} \sum_{j_3=0}^{n-1} \coeffFace{i}{j_1}{j_2}{j_3} 
 \phiFace{i}{j_1}{j_2}{j_3}(\f{x}), \quad \coeffFace{i}{j_1}{j_2}{j_3} \in \R,
\]
for the case of a boundary face~$\face^{(i)}$, $i \in \indexFaceBoundary$. 

\begin{figure}
\begin{center}
\resizebox{0.45\textwidth}{!}{
 \begin{tikzpicture}
  \coordinate(A) at (-0.5,0);  \coordinate(C) at (1.4,0.5); \coordinate(G) at (1.3,4); \coordinate(E) at (0.2,2.7);
  \coordinate(B) at (3,-0.9); \coordinate(D) at (4,0.4); \coordinate(H) at (4.3,3.8); \coordinate(F) at (2.9,3);
  
   \fill[gray!10] (C) -- (G) -- (G) -- (E) -- (E) -- (A) -- (A) -- (C);  
   \fill[gray!10] (A) -- (B) -- (B) -- (F) -- (F) -- (E) -- (E) -- (A); 
  
  \draw[dashed] (C) -- (A); \draw[dashed] (C) -- (G);  \draw (G) -- (E); \draw (E) -- (A);

  \draw (B) -- (A); \draw (B) -- (D);  \draw[dashed] (D) -- (C); \draw (D) -- (H); \draw (G) -- (H); \draw (F) -- (H); \draw (F) -- (E); 
  \draw (F) -- (B);

  \fill (A) circle (1pt); \fill (B) circle (1pt); \fill (C) circle (1pt); \fill (D) circle (1pt);
  \fill (E) circle (1pt); \fill (F) circle (1pt); \fill (G) circle (1pt); \fill (H) circle (1pt); 

  \draw[->] (-0.05,0.05) -- (1,0.33);
  \draw[->] (0,-0.05) -- (1,-0.3);
  \draw[->] (-0.37,0.15) -- (-0.1,1.2);

  \node at (1.1,0.2) {\scriptsize $\xi_2$};
  \node at (0.1,1) {\scriptsize $\xi_3$};
  \node at (0.9,-0.1) {\scriptsize $\xi_1$};
  \node[below] at (A) {\scriptsize $\vertex^{(i)}$}; %
  \node at (2.2,1.8) {\scriptsize $\patch^{(i_m)}$};
  \node at (-0.5,1.5) {\scriptsize $\edge^{(i)}$};
  \node at (0.75,1.4) {\scriptsize $\face^{(\LeftFace{i}{i_m})}$};
  \node at (2.1,1.0) {\scriptsize $\face^{(\RightFace{i}{i_m})}$};
  \end{tikzpicture}
 }
\end{center}
\caption{The parameterization~$\map^{(i_m)}$ of the patch~$\overline{\patch^{(i_m)}}$ is called to be in standard form with respect to the 
edge~$\edge^{(i)}$ or with respect to the vertex~$\vertex^{(i)}$. For the edge~$\edge^{(i)}$ of the patch~$\overline{\patch^{(i_m)}}$, the 
two neighboring faces are denoted by $\face^{(\LeftFace{i}{i_m})}$ and $\face^{(\RightFace{i}{i_m})}$ as 
shown in the figure.} \label{fig:standard_form_edge_functions}
\end{figure}

For a patch~$\patch^{(i_m)}$, $i_m \in \indexPatch$, we say that the associated geometry mapping~$\map^{(i_m)}$ is given in 
\textit{standard form} with respect to an edge~$\edge^{(i)}$, $i \in \indexEdge$, $\edge^{(i)} \subset \overline{\patch^{(i_m)}}$, or with respect to 
a vertex~$\vertex^{(i)}$, $i \in \indexVertex$, $\vertex^{(i)} \in \overline{\patch^{(i_m)}}$, when the geometry mapping~$\map^{(i_m)}$ is 
parameterized as shown in Fig.~\ref{fig:standard_form_edge_functions}. Note that the geometry mapping~$\map^{(i_m)}$ can always be reparameterized 
(if necessary) to be in standard form with respect to an edge or a vertex.

We denote for each edge~$\edge^{(i)}$, $i \in \indexEdge$, by $\WhichPatchesE{i}$ the set of the indices~$i_m$ of those patches~$\patch^{(i_m)}$, 
$i_m \in \indexPatch$, whose closure contains the edge~$\edge^{(i)}$, i.e. $\edge^{(i)} \subset \overline{\patch^{(i_m)}}$, and define 
for each vertex~$\vertex^{(i)}$, $i \in \indexVertex$, the set~$\WhichPatchesV{i}$, which collects the indices~$i_m$, $i_m \in \indexPatch$, 
such that $\vertex^{(i)} \in \overline{\patch^{(i_m)}}$, and the set~$\WhichEdgesV{i}$, which collects the indices~$i_s$ 
of the three edges~$\edge^{(i_s)}$, $i_s \in \indexEdge$, such that $\vertex^{(i)} \in \overline{\edge^{(i_s)}}$. 
For each edge~$\edge^{(i)}$, $i \in \indexEdge$, 
assuming without loss of generality that the geometry mappings~$\map^{(i_m)}$ of the patches~$\patch^{(i_m)}$, $i_m \in \WhichPatchesE{i}$, are 
in standard form with respect to the edge~$\edge^{(i)}$, cf.
Fig.~\ref{fig:standard_form_edge_functions}, let $\phiOneEdgeAll{i}: \overline{\domain} \rightarrow \R$ be 
the isogeometric function defined as the linear combination of standard isogeometric functions~$\phiPatch{i_m}{j_1}{j_2}{j_3}$ in the vicinity of the 
edge~$\edge^{(i)}$, more precisely
\[
 \phiOneEdgeAll{i}(\f{x}) = \sum_{i_m \in \WhichPatchesE{i}} \sum_{j_1=0}^{1} \sum_{j_2=0}^{1} \sum_{j_3=0}^{n-1} 
 \coeffPatchEdgeSt{i}{j_1}{j_2}{j_3}{i_m} \phiPatch{i_m}{j_1}{j_2}{j_3}(\f{x}),
\]
with $\coeffPatchEdgeSt{i}{j_1}{j_2}{j_3}{i_m} \in \R$. Similarly, we define for each vertex~$\vertex^{(i)}$, $i \in \indexVertex$, now assuming 
without loss of generality that the geometry mappings~$\map^{(i_m)}$ of the patches~$\patch^{(i_m)}$, $i_m \in \WhichPatchesV{i}$, are 
in standard form with respect to the vertex~$\vertex^{(i)}$, cf. Fig.~\ref{fig:standard_form_edge_functions}, the isogeometric function given by
\[
 \phiOneVertexAll{i}(\f{x}) = \sum_{i_m \in \WhichPatchesV{i}} \sum_{j_1=0}^{1} \sum_{j_2=0}^{1} \sum_{j_3=0}^{1} 
 \coeffPatchVertex{i}{j_1}{j_2}{j_3}{i_m} \phiPatch{i_m}{j_1}{j_2}{j_3}(\f{x}),
\]
with $\coeffPatchVertex{i}{j_1}{j_2}{j_3}{i_m} \in \R$, which is the linear combination of standard isogeometric 
functions~$\phiPatch{i_m}{j_1}{j_2}{j_3}$ in the neighborhood of the vertex~$\vertex^{(i)}$. We further denote by~$\coeffAllFace$ the vector of all 
coefficients~$\coeffFace{i}{j_1}{j_2}{j_3}$, $\coeffPatchEdgeSt{i}{j_1}{j_2}{j_3}{i_m}$ and $\coeffPatchVertex{i}{j_1}{j_2}{j_3}{i_m}$ of the 
isogeometric functions $\phiFaceAll{i}$, $i \in \indexFace$, $\phiOneEdgeAll{i}$, $i \in \indexEdge$, and $\phiOneVertexAll{i}$, $i \in \indexVertex$, 
respectively. 

For each edge~$\edge^{(i)}$, $i \in \indexEdge$, and patch~$\patch^{(i_m)}$, $i_m \in \WhichPatchesE{i}$, assuming that the associated geometry mapping~$\map^{(i_m)}$ is given in standard form with respect to the edge~$\edge^{(i)}$, 
we define by $\LeftFace{i}{i_m}$ the index $\ell \in \indexFace$ for which 
\[
 \face^{(\ell)} \subset \map^{(i_m)}(\{ 0\}  \times [0,1] \times [0,1]),
\]
and similarly by $\RightFace{i}{i_m}$ the index $\ell' \in \indexFace$ for which
\[
 \face^{(\ell')} \subset \map^{(i_m)}([0,1] \times \{ 0\} \times [0,1]),
\]
cf. Fig.~\ref{fig:standard_form_edge_functions}.

Recall that for any face~$\face^{(i)}$, $i \in \indexFace$, all functions $\phiFace{i}{j_1}{j_2}{j_3}$, or equivalently all possible variations 
of $\phiFaceAll{i}$, span the space of those isogeometric functions~$\phi \in \V$, which are $C^1$ at the face~$\face^{(i)}$ and which possess a 
support limited to the vicinity of the face~$\face^{(i)}$, or more precisely, which have a support with respect to the standard isogeometric 
spline functions~$\phiPatch{\ell}{j_1}{j_2}{j_3}$ with non-vanishing values or non-vanishing gradients at the face~$\face^{(i)}$, cf. 
Section~\ref{subsec:preliminaries} and \ref{subsec:two_patch_case}.

Therefore, the space~$\SpaceFaceAll$ is equal to the space of all functions that are formed by a linear combination of functions~$\phiFaceAll{i}$, 
$i \in \indexFace$, which are compatible (i.e. coincide) at their possible common supports in the neighborhood of the edges~$\edge^{(\ell)}$, 
$\ell \in \indexEdge$, and vertices~$\vertex^{(\ell)}$, $\ell \in \indexVertex$, of the multi-patch volume~$\overline{\domain}$, and by subtracting 
those standard isogeometric spline functions~$\phiPatch{\ell'}{j_1}{j_2}{j_3}$ which have been added to often. By studying the possible 
common supports of the functions~$\phiFaceAll{i}$, $i \in \indexFace$, at the edges and vertices of the multi-patch volume~$\overline{\domain}$, we
observe that such a function possesses the form
\begin{equation} \label{eq:functionComplexAll}
 \sum_{i \in \indexFace} \phiFaceAll{i} - \sum_{i \in \indexEdge} \phiOneEdgeAll{i} - 2\sum_{i \in \indexVertex} \phiOneVertexAll{i},
\end{equation}
where the coefficients $\coeffAllFace$ have to satisfy 
\begin{equation} \label{eq:facesequal}
 \partial_{1}^{\ell_1} \partial_{2}^{\ell_2} \big(\phiFaceAll{\LeftFace{i}{i_m}} \circ \map^{(i_m)} \big)(0,0,\xi) = 
 \partial_{1}^{\ell_1} \partial_{2}^{\ell_2} \big(\phiFaceAll{\RightFace{i}{i_m}} \circ \map^{(i_m)} \big)(0,0,\xi),
 \mbox{ } 0 \leq \ell_1,\ell_2, \leq 1, \, \xi \in [0,1],
\end{equation}
and 
\begin{equation} \label{eq:facesedgesequal}
 \partial_{1}^{\ell_1} \partial_{2}^{\ell_2} \big(\phiFaceAll{\LeftFace{i}{i_m}} \circ \map^{(i_m)} \big)(0,0,\xi) = 
 \partial_{1}^{\ell_1} \partial_{2}^{\ell_2} \big(\phiOneEdgeAll{i} \circ \map^{(i_m)} \big)(0,0,\xi),
 \mbox{ } 0 \leq \ell_1,\ell_2, \leq 1, \, \xi \in [0,1],
\end{equation}
for each edge~$\edge^{(i)}$, $i \in \indexEdge$, and patch~$\patch^{(i_m)}$, $i_m \in 
\WhichPatchesE{i}$, and  
\begin{equation} \label{eq:edgesverticesequal}
 \partial_{1}^{\ell_1} \partial_{2}^{\ell_2} \partial_{3}^{\ell_3}  \big(\phiOneVertexAll{i} \circ \map^{(i_m)} \big)(0,0,0) = 
 \partial_{1}^{\ell_1} \partial_{2}^{\ell_2} \partial_{3}^{\ell_3}  \big(\phiOneEdgeAll{{ i_s}} \circ \map^{(i_m)} \big)(0,0,0),
 \mbox{ }0 \leq \ell_1,\ell_2,\ell_3 \leq 1,
\end{equation}
for each vertex~$\vertex^{(i)}$, $i \in \indexVertex$, patch~$\patch^{(i_m)}$, $i_m \in \WhichPatchesV{i}$, and edge~$\edge^{({ i_s})}$, ${i_s} \in 
\WhichEdgesV{i}$, assuming that in each case the geometry mapping~$\map^{(i_m)}$ is given in standard form with respect to the corresponding 
edge~$\edge^{(i)}$ or vertex~$\vertex^{(i)}$, cf. Fig.~\ref{fig:standard_form_edge_functions}. While condition~\eqref{eq:facesequal} guarantees 
that the single functions~$\phiFaceAll{i}$ are compatible in the vicinity of the edges and vertices of the multi-patch domain~$\overline{\domain}$, 
conditions~\eqref{eq:facesedgesequal} and \eqref{eq:edgesverticesequal} ensure that the correct multiples of the standard isogeometric spline 
functions~$\phiPatch{\ell'}{j_1}{j_2}{j_3}$ are subtracted, which all together implies that the resulting function~\eqref{eq:functionComplexAll} is 
$C^1$ on $\overline{\domain}$ and therefore belongs to the space~$\V^1$.

Equations~\eqref{eq:facesequal} and \eqref{eq:facesedgesequal} are equivalent to
\begin{equation} \label{eq:facesequalNew}
 \partial_{1}^{\ell_1} \partial_{2}^{\ell_2} \big(\phiFaceAll{\LeftFace{i}{i_m}} \circ \map^{(i_m)} \big)(0,0,\zeta_j) = 
 \partial_{1}^{\ell_1} \partial_{2}^{\ell_2} \big(\phiFaceAll{\RightFace{i}{i_m}} \circ \map^{(i_m)} \big)(0,0,\zeta_j), 
 \mbox{ } 0 \leq \ell_1,\ell_2, \leq 1, \, j \in \I,
\end{equation}
and
\begin{equation} \label{eq:facesedgesequalNew}
  \partial_{1}^{\ell_1} \partial_{2}^{\ell_2} \big(\phiFaceAll{\LeftFace{i}{i_m}} \circ \map^{(i_m)} \big)(0,0,\zeta_j) = 
 \partial_{1}^{\ell_1} \partial_{2}^{\ell_2} \big(\phiOneEdgeAll{i} \circ \map^{(i_m)} \big)(0,0,\zeta_j), 
 \mbox{ } 0 \leq \ell_1,\ell_2, \leq 1, \, j \in \I,
\end{equation}
respectively, where $\zeta_j$, $j \in \{0,1,\ldots,n-1\}$, are the Greville abscissae with respect to the univariate spline space~$\US{p}{r}{h}$. Then, 
all equations~\eqref{eq:edgesverticesequal}, \eqref{eq:facesequalNew} and \eqref{eq:facesedgesequalNew} build a homogeneous system of linear equations
\begin{equation} \label{eq:homogoneous_system_Large}
 \TMatrixFace \coeffAllFace = \f{0},
\end{equation}
for the coefficients~$\coeffAllFace$, and any choice of the coefficient vector~$\coeffAllFace$, which fulfills the linear 
system~\eqref{eq:homogoneous_system_Large}, specifies an isogeometric function~\eqref{eq:functionComplexAll} which belongs to the space~$\V^1$. 
This allows us to describe the space~$\SpaceFaceAll$ as 
\[
 \SpaceFaceAll = \left\{ \sum_{i \in \indexFace} \phiFaceAll{i} - \sum_{i \in \indexEdge} \phiOneEdgeAll{i} - 2\sum_{i \in \indexVertex} \phiOneVertexAll{i}
 \; | \; \coeffAllFace \in \R^{|\coeffAllFace|}, \; \TMatrixFace \coeffAllFace = \f{0} \right\} .
\]

Note that the selected strategy to generate $C^1$ isogeometric spline functions across the patch faces is inspired by the construction of $C^1$ and 
$C^2$ isogeometric spline functions in the vicinity of a vertex of a planar multi-patch domain presented in~\cite{KaSaTa19a,KaSaTa19b} and 
\cite{KaVi19a,KaVi20}, respectively. There, $C^1$/$C^2$ isogeometric spline functions are generated in the neighborhood of a vertex as the sum of 
compatible edge functions for the single edges and by subtracting those standard isogeometric spline functions which have been added twice. 
In \cite{KaVi20b}, this approach has been generalized to the case of $C^s$ isogeometric spline functions over planar 
multi-patch parameterizations for an abritrary $s \geq 1$.

Now, analyzing the conditions~\eqref{eq:edgesverticesequal}, \eqref{eq:facesequalNew} and \eqref{eq:facesedgesequalNew}, we observe that 
a coefficient $\coeffFace{i}{j_1}{j_2}{j_3}$ is not involved in these equations if $j_1=0,1$, $j_2,j_3 \in \I_{j_1} \setminus 
\{0,\ldots,2-j_1,n_{j_1}-3+j_1,\ldots,n_{j_1}-1 \}$, for the case of an inner face~$\face^{(i)}$, $i \in \indexFaceInner$, and if $j_1=0,1$, 
$j_2,j_3 \in \I \setminus \{0,1,n-2,n-1 \}$, for the case of a boundary face $\face^{(i)}$, $i \in \indexFaceBoundary$, which simplifies 
the homogeneous linear system~\eqref{eq:homogoneous_system_Large} to the system  
\begin{equation} \label{eq:homogoneous_system}
 \TMatrixEdge \coeffAllEdge = \f{0},
\end{equation}
where~$\coeffAllEdge$ is the vector of all coefficients~$\coeffAllFace$ which are involved in the equations \eqref{eq:edgesverticesequal}, 
\eqref{eq:facesequalNew} and \eqref{eq:facesedgesequalNew}. This is a direct consequence of the fact that the function~$\phiFace{i}{j_1}{j_2}{j_3}$ 
for $j_1=0,1$, $j_2,j_3 \in \I_{j_1} \setminus \{0,\ldots,2-j_1,n_{j_1}-3+j_1,\ldots,n_{j_1}-1 \}$, $i \in \indexFaceInner$, and for $j_1=0,1$, 
$j_2,j_3 \in \I \setminus \{0,1,n-2,n-1 \}$, $i \in \indexFaceBoundary$, has vanishing values and gradients at all other 
faces~$\overline{\face^{(\ell)}}$, $\ell \in \indexFace \setminus \{ i \}$, i.e. $\phiFace{i}{j_1}{j_2}{j_3}(\overline{\face^{(\ell)}})=0$ and 
$\nabla \phiFace{i}{j_1}{j_2}{j_3}(\overline{\face^{(\ell)}})=\f{0}$, but which also further implies that the corresponding 
function~$\phiFace{i}{j_1}{j_2}{j_3}$ is $C^1$ on the entire multi-patch volume~$\overline{\domain}$, and therefore belongs to the space~$\V^1$, see 
also Section~\ref{subsec:preliminaries}. Therefore, the space~$\SpaceFaceAll$ can be decomposed into the direct sum
\[
 \SpaceFaceAll = \left( \bigoplus_{i \in \indexFace} \SpaceFace{i} \right) \oplus \SpaceEdgeAll,
\]
with the single subspaces~$\SpaceFace{i}$ given by
\begin{equation} \label{eq:facespaceinner}
 \SpaceFace{i} = \Span \{  \phiFace{i}{j_1}{j_2}{j_3}  \,|\, j_1=0,1, \, j_2,j_3 \in \I_{j_1} \setminus 
 \{0,\ldots,2-j_1,n_{j_1}-3+j_1,\ldots,n_{j_1}-1 \} \},
 \end{equation}
for an inner face~$\face^{(i)}$, $i \in \indexFaceInner$, and by
\begin{equation} \label{eq:facespaceboundary}
\SpaceFace{i} = \Span \left\{ \phiFace{i}{j_1}{j_2}{j_3}  \;|\; j_1 =0,1, \, j_2,j_3 \in \I \setminus \{0,1,n-2,n-1 \} \right\},
\end{equation}
for a boundary face~$\face^{(i)}$, $i \in \indexFaceBoundary$, and with the subspace~$\SpaceEdgeAll$ which is equal to 
\begin{equation} \label{eq:edgespaceAll}
 \SpaceEdgeAll = \left\{ \sum_{i \in \indexFace} \phiFaceAllRes{i} - \sum_{i \in \indexEdge} \phiOneEdgeAll{i} - 2\sum_{i \in \indexVertex} 
 \phiOneVertexAll{i} \; | \; \coeffAllEdge \in \R^{|\coeffAllEdge|}, \; \TMatrixEdge \coeffAllEdge = \f{0} \right\},
\end{equation}
where $\phiFaceAllRes{i}$ is the function~$\phiFaceAll{i}$ for which the coefficients~$\coeffFace{i}{j_1}{j_2}{j_3}$ are set to zero if $j_1=0,1$, 
$j_2,j_3 \in \I_{j_1} \setminus \{0,\ldots,2-j_1,n_{j_1}-3+j_1,\ldots,n_{j_1}-1 \}$ for $i \in \indexFaceInner$, and if $j_1=0,1$, 
$j_2,j_3 \in \I \setminus \{0,1,n-2,n-1 \}$, for $i \in \indexFaceBoundary$. 

Let $n_{\TMatrixEdge}$ be the dimension of the kernel of the matrix~$\TMatrixEdge$ in the homogeneous system of linear 
equations~\eqref{eq:homogoneous_system}, i.e. $n_{\TMatrixEdge} = \dim \ker (\TMatrixEdge)$. Each basis of the $\ker (\TMatrixEdge)$ determines 
$n_{\TMatrixEdge}$ linearly independent $C^1$ isogeometric functions, which we will denote by $\phiEdgeAllSingle{j}$, 
$j \in \{0,1,\ldots,n_{\TMatrixEdge}-1 \}$, and which form a basis of the space~$\SpaceEdgeAll$, i.e.
\[
 \SpaceEdgeAll = \Span \left\{ \phiEdgeAllSingle{j} \;|\; j \in \{0,1,\ldots,n_{\TMatrixEdge} -1\} \right\}.
\]
A possible strategy to compute a basis for the $\ker (\TMatrixEdge)$ is to use the concept of minimal determining sets (cf.~\cite{LaSch07}) for the 
coefficients~$\coeffAllEdge$. An example of such a minimal determining set algorithm, which can be directly applied to our configuration, 
{is described in~\cite{KaVi17a}}. While the functions~$\phiPatch{i}{j_1}{j_2}{j_3}$ and $\phiFace{i}{j_1}{j_2}{j_3}$ of the spline spaces 
$\SpacePatch{i}$, $i \in \indexPatch$, and $\SpaceFace{i}$, $i \in \indexFace$, are locally supported by their definition and construction, this 
is in general not true for the resulting functions $\phiEdgeAllSingle{j}$, $j \in \{0,1,\ldots,n_{\TMatrixEdge}-1\}$, which can be in the worst case 
even supported over all or over most edges of the multi-patch volume~$\overline{\domain}$. However, the method allows a significant reduction of 
the support of the generated functions~$\phiEdgeAllSingle{j}$, $j \in \{0,1,\ldots,n_{\TMatrixEdge}-1\}$, by an appropriate separation of the linear 
system~\eqref{eq:homogoneous_system} and by a careful preselection of some coefficients of $\coeffAllEdge$. E.g. in~\cite{KaVi17b}, the minimal 
determining set algorithm~\cite{KaVi17a} was used and adapted to generate $C^2$ functions over edges of planar bilinearly parameterized multi-patch 
domains, which are just supported over one edge or over the edges containing one particular vertex.

Summarized, we obtain:
\begin{thm} \label{thm:basis}
The $C^1$ isogeometric spline space~$\V^1$ over the trilinear multi-patch volume~$\overline{\domain}$ can be decomposed into the direct sum
\begin{equation} \label{eq:direct_sum_GeneralMultiPatch}
 \V^1= \left( \bigoplus_{i \in \indexPatch} \SpacePatch{i} \right) \oplus 
 \left( \bigoplus_{i \in \indexFace} \SpaceFace{i} \right) \oplus \SpaceEdgeAll,
\end{equation}
where the single subspaces~$\SpacePatch{i}$, $i \in \indexPatch$, $\SpaceFace{i}$, $i \in \indexFaceInner$, $\SpaceFace{i}$, $i \in \indexFaceBoundary$, 
and $\SpaceEdgeAll$ are given by \eqref{eq:patch_space}, \eqref{eq:facespaceinner}, \eqref{eq:facespaceboundary} and \eqref{eq:edgespaceAll}, 
respectively. Moreover, the functions $\phiPatch{i}{j_1}{j_2}{j_3}$, $j_1,j_2,j_3 \in \I \setminus \{0,1,n-2,n-1 \}$, $i \in \indexPatch$, $\phiFace{i}{j_1}{j_2}{j_3}$, 
$j_1=0,1$, $j_2,j_3 \in \I_{j_1} \setminus \{0,\ldots,2-j_1,n_{j_1}-3+j_1,\ldots,n_{j_1}-1 \}$, $i \in \indexFaceInner$, 
$\phiFace{i}{j_1}{j_2}{j_3}$, $j_1 =0,1$, $j_2,j_3 \in \I \setminus \{0,1,n-2,n-1 \}$, $i \in \indexFaceBoundary$, and 
$\phiEdgeAllSingle{j}$, $j \in \{0,\ldots,n_{T}-1 \}$ of the spaces $\SpacePatch{i}$, $i \in \indexPatch$, $\SpaceFace{i}$, $i \in \indexFaceInner$, 
$\SpaceFace{i}$, $i \in \indexFaceBoundary$,  and $\SpaceEdgeAll$, respectively, form a basis of the $C^1$ isogeometric spline space~$\V^1$. 
\end{thm}
\begin{proof}
The equivalence~\eqref{eq:direct_sum_GeneralMultiPatch} as well as the claim that the functions $\phiPatch{i}{j_1}{j_2}{j_3}$, 
$\phiFace{i}{j_1}{j_2}{j_3}$ and $\phiEdgeAllSingle{j}$ of the spline spaces $\SpacePatch{i}$, $i \in \indexPatch$, $\SpaceFace{i}$, $i \in \indexFace$, 
and $\SpaceEdgeAll$ form a basis of the $C^1$ space~$\V^1$ directly follow from the construction of the space~$\V^1$ presented above.
\end{proof}

Since the $C^1$ isogeometric spline space~$\V^1$ is the direct sum~\eqref{eq:direct_sum_GeneralMultiPatch}, the dimension of $\V^1$ is equal to
\[
 \dim \V^1 = \sum_{i \in \indexPatch} \dim \SpacePatch{i}  +  \sum_{i \in \indexFace} \dim \SpaceFace{i}  + \dim \SpaceEdgeAll .
\]
While the dimensions of the spaces $\SpacePatch{i}$, $i \in \indexPatch$, and of the spaces~$\SpaceFace{i}$, $i \in \indexFace$, just depend on 
the degree~$p$, the regularity~$r$ and the number of spline elements, i.e. $k+1$, of the underlying spline space~$\US{p}{r}{h}$, and are simply given as
\[
 \dim \SpacePatch{i} = (|\I|-4)^3, 
\]
and 
\[
 \dim \SpaceFace{i} = \begin{cases}
  2 (|\I|-4)^2  & \mbox{if }i \in \indexFaceBoundary, \\
  (|\I_0|-6)^2 + (|\I_1|-4)^2  & \mbox{if }i \in \indexFaceInner,
 \end{cases}
\]
the dimension of the space~$\SpaceEdgeAll$, that is the number $n_{\TMatrixEdge}$, also depends on the number of patches, faces, edges and vertices of 
the multi-patch volume~$\overline{\domain}$, and further depends on the valencies
of the single edges and vertices, and on the shapes of the individual trilinearly parameterized patches. A detailed study of the dimension of 
$\SpaceEdgeAll$ is beyond the scope of the paper and is the topic of possible future research. However, the numerical investigation of its dimension for a specific subclass of trilinearly parameterized multi-patch volumes will be presented in Section~\ref{subsec:dimension}.

\begin{rem} \label{rem:modified_construction}
The presented construction of the $C^1$ isogeometric spline space over the trilinear multi-patch volume~$\overline{\domain}$ is a general framework for 
the uniform design of the space~$\V^1$ and of a basis of the space for any possible configuration of the trilinear multi-patch 
volume~$\overline{\domain}$.  But clearly, the selected splitting of the space~$\V^1$ into the direct sum of simpler subspaces is not the only possible 
one. E.g., in case of a boundary face or in case of a boundary edge with a patch valency~$1$, the corresponding patch space~$\SpacePatch{i}$, 
$i \in \indexPatch$, could be trivially extended to the boundary face or to the boundary edge to obtain a slightly modified patch 
space~$\SpaceFaceMod{i}$. Similarly, in case of a boundary edge with patch valency~$2$, the corresponding face space~$\SpaceFace{i}$, 
$i \in \indexFaceInner$, could be trivially enlarged to the boundary edge to get a slightly adapted face space~$\SpaceFaceMod{i}$. The described 
steps would then also lead to a modified edge space~$\SpaceEdgeAllMod$, which would be smaller and simpler such as for the particular subclass of 
trilinear multi-patch volumes~$\overline{\domain}$ considered in Section~\ref{subsec:example_one_inner_edge}, or which would even vanish like in the two patch 
case in Section~\ref{subsec:two_patch_case}.
\end{rem}

%%%%%%%%%%%%%%%%%%%%%%%%%%%%%%%%%%%%%%%%%%%%%%%%%%%%%%%%%%%%%%%%%%%%%%%%%%%%%%%%%%%%%%%%%%%%%
\section{A specific subclass of trilinear multi-patch volumes} \label{sec:specific-class}
%%%%%%%%%%%%%%%%%%%%%%%%%%%%%%%%%%%%%%%%%%%%%%%%%%%%%%%%%%%%%%%%%%%%%%%%%%%%%%%%%%%%%%%%%%%%%%%%%%%%%

In this section, we describe for a particular subclass of trilinear multi-patch volumes the above presented method for the design of the 
$C^1$ isogeometric spline space~$\V^1$ and of an associated basis in more detail. In addition, we numerically compute the dimension of the resulting 
$C^1$ spline space~$\V^1$ and perform $ L^2$ approximation to numerically investigate the approximation power of this space.

%%%%%%%%%%%%%%%%%%%%%%%%%%%%%%%%%%%%%%%%%%%%%%%%%%%%%%%%%%%%%%%%%%%%%%%%%%%%%%%%%%%%%%%%%%%%%%%%%%%%%%%%%%%%%%%%%%%%%%%
\subsection{The subclass of trilinear multi-patch volumes with one inner edge} \label{subsec:example_one_inner_edge}
%%%%%%%%%%%%%%%%%%%%%%%%%%%%%%%%%%%%%%%%%%%%%%%%%%%%%%%%%%%%%%%%%%%%%%%%%%%%%%%%%%%%%%%%%%%%%%%%%%%%%%%%%%%%%%%%%%%%%%%

In the following, let us consider a particular subclass of trilinear multi-patch volumes~$\overline{\domain}$, denoted by $\classVolume$, 
where each of the multi-patch volumes is the union of $\numPatches$ patches~$\overline{\patch^{(i)}}$, $i \in \indexPatch$, with 
$\indexPatch = \{ 0,1,\ldots, \nu-1\}$ and $\numPatches \geq 3$, that is $\overline{\domain} = \cup_{i=0}^{\numPatches-1} \overline{\patch^{(i)}}$, 
possesses $\numPatches$ inner faces~$\face^{(i)}$, $i \in \indexFaceInner$, with $\indexFaceInner = \{0,1,\ldots,\numPatches-1 \}$, and has exactly one 
inner edge labeled by $\edge^{(0)}$. For each example of such a trilinear multi-patch volume~$\overline{\domain}$, we assume that the inner 
faces~$\overline{\face^{(i)}}$, $i \in \{0,1,\ldots,\numPatches-1 \}$, are given by $\overline{\face^{(i)}} = \overline{\patch^{(i)}} 
\cap \overline{\patch^{(i+1)}} $, or more precisely by
\[
\map^{(i)}(0,t_1,t_2) = \map^{(i+1)}(t_1,0,t_2), 
\quad (t_1,t_2) \in [0,1]^2,
\]
considering the upper index~$i$ modulo~$\nu$, and that the inner edge~$\overline{\edge^{(0)}}$ is given by $\overline{\edge^{(0)}} = 
\cap_{i=0}^{\numPatches-1} \overline{\patch^{(i)}}$, or more precisely by
\[
 \map^{(0)}(0,0,t) = \cdots =  \map^{(\numPatches-1)}(0,0,t),
 \quad t \in [0,1],
\]
see Fig.~\ref{fig:MultiPatchVolume_InnerEdge} for the case of a trilinear three- and four-patch volume~$\overline{\domain}$.
This means that the geometry mappings~$\map^{(i)}$, $i \in \{0,1,\ldots,\numPatches-1\}$, are given in standard form with respect to the inner 
faces~$\face^{(i)}$, $i \in \{0,1,\ldots,\numPatches-1 \}$, and with respect to the inner edge~$\edge^{(0)}$.

\begin{figure}
\begin{center}
\begin{tabular}{cc}
\resizebox{0.43\textwidth}{!}{
 \begin{tikzpicture}
  \coordinate(A) at (0,0);  \coordinate(B) at (1.5,1.0); \coordinate(C) at (-0.4,2.2); \coordinate(D) at (-1.8,0.6);
  \coordinate(E) at (-1.2,-1.2); \coordinate(F) at (1.0,-1.6); \coordinate(G) at (2.8,-0.6); 
  
  \coordinate(H) at (0.2,3.2);  \coordinate(I) at (1.9,3.8); \coordinate(J) at (-0.6,4.8); \coordinate(K) at (-1.6,3.4);
  \coordinate(L) at (-1.2,1.8); \coordinate(M) at (1.0,1.4); \coordinate(N) at (3.0,2.0); 
  \fill[gray!10] (A) -- (F) -- (F) -- (M) -- (M) -- (H) -- (H) -- (A);  
  \fill[gray!10] (A) -- (B) -- (B) -- (I) -- (I) -- (H) -- (H) -- (A); 
  \fill[gray!10] (A) -- (D) -- (D) -- (K) -- (K) -- (H) -- (H) -- (A); 
  
  \draw[dashed] (A) -- (B); \draw[dashed] (A) -- (D);  \draw[dashed] (A) -- (F); 
  \draw[dashed] (B) -- (C); \draw[dashed] (C) -- (D); \draw (D) -- (E); \draw (E) -- (F); \draw (F) -- (G); \draw[dashed] (G) -- (B);
  
  \draw (H) -- (I); \draw (H) -- (K);  \draw (H) -- (M); 
  \draw (I) -- (J); \draw (J) -- (K); \draw (K) -- (L); \draw (L) -- (M); \draw (M) -- (N); \draw (N) -- (I);

  \draw[dashed] (A) -- (H); \draw[dashed] (B) -- (I); \draw[dashed] (C) -- (J); \draw (D) -- (K); \draw (E) -- (L); \draw (F) -- (M); 
  \draw (G) -- (N);

  \fill (A) circle (1pt); \fill (B) circle (1pt); \fill (C) circle (1pt); \fill (D) circle (1pt);
  \fill (E) circle (1pt); \fill (F) circle (1pt); \fill (G) circle (1pt); \fill (H) circle (1pt); 
  \fill (I) circle (1pt); \fill (J) circle (1pt); \fill (K) circle (1pt); \fill (L) circle (1pt);
  \fill (M) circle (1pt); \fill (N) circle (1pt); 

  \draw[->] (0.2,0) -- (0.9,0.45);
  \draw[->] (0.05,0.2) -- (0.75,0.65);
  \draw[->] (0.22,-0.15) -- (0.7,-0.9);
  \draw[->] (-0.03,-0.15) -- (0.45,-0.9);
  \draw[->] (-0.13,-0.1) -- (-0.9,0.18);
  \draw[->] (-0.13,0.16) -- (-0.9,0.44);
  \draw[->] (-0.1,0.25) -- (-0.05,1.4);

  \node at (0.86,0.15) {\scriptsize $\xi_2$};
  \node at (0.5,0.75) {\scriptsize $\xi_1$};
  \node at (-0.7,0.6) {\scriptsize $\xi_2$};
  \node at (-0.75,-0.05) {\scriptsize $\xi_1$};
  \node at (0.15,-0.8) {\scriptsize $\xi_2$};
  \node at (0.8,-0.6) {\scriptsize $\xi_1$};
  \node at (-0.25,1) {\scriptsize $\xi_3$};
  \node at (-0.7,-0.7) {\scriptsize $\patch^{(1)}$};
  \node at (1.9,-0.5) {\scriptsize $\patch^{(2)}$};
  \node at (0,4.1) {\scriptsize $\patch^{(0)}$};
  \node at (0.5,1.2) {\scriptsize $\face^{(1)}$};
  \node at (-0.75,2.4) {\scriptsize $\face^{(0)}$};
  \node at (0.4,2.0) {\scriptsize $\edge^{(0)}$};
  \node at (1.3,2.8) {\scriptsize $\face^{(2)}$};
  \node at (-0.25,-0.3) {\scriptsize $\vertex^{(0)}$};
  \node at (1.85,1.05) {\scriptsize $\vertex^{(1)}$};
  \node at (-0.1,2.35) {\scriptsize $\vertex^{(2)}$};
  \node at (-2.1,0.65) {\scriptsize $\vertex^{(3)}$};
  \node at (-1.3,-1.4) {\scriptsize $\vertex^{(4)}$};
  \node at (1.0,-1.8) {\scriptsize $\vertex^{(5)}$};
  \node at (2.8,-0.85) {\scriptsize $\vertex^{(6)}$};
  \node at (0.15,3.4) {\scriptsize $\vertex^{(7)}$};
  \node at (1.95,4.05) {\scriptsize $\vertex^{(8)}$};
  \node at (-0.4,5.0) {\scriptsize $\vertex^{(9)}$};
  \node at (-2.0,3.4) {\scriptsize $\vertex^{(10)}$};
  \node at (-1.3,1.5) {\scriptsize $\vertex^{(11)}$};
  \node at (1.3,1.7) {\scriptsize $\vertex^{(12)}$};
  \node at (3.3,2.3) {\scriptsize $\vertex^{(13)}$};
  \end{tikzpicture}
 }
 &
 \resizebox{0.41\textwidth}{!}{
 \begin{tikzpicture}
  \coordinate(A) at (0,0);  \coordinate(B) at (2.0,-0.1); \coordinate(C) at (1.7,1.3); \coordinate(D) at (0.5,1.4);
  \coordinate(E) at (-1.1,1.3); \coordinate(F) at (-1.7,-0.2); \coordinate(G) at (-1.3,-1.3); \coordinate(H) at (0.4,-1.7); 
  \coordinate(I) at (1.7,-1.5);

   \coordinate(J) at (0,2.2);  \coordinate(K) at (2.0,2.1); \coordinate(L) at (1.7,3.3); \coordinate(M) at (0.5,3.6);
  \coordinate(N) at (-1.2,3.3); \coordinate(O) at (-1.7,2.0); \coordinate(P) at (-1.4,0.9); \coordinate(Q) at (0.4,0.5); 
  \coordinate(R) at (1.7,0.7); 
  \fill[gray!10] (A) -- (B) -- (B) -- (K) -- (K) -- (J) -- (J) -- (A);  
  \fill[gray!10] (A) -- (D) -- (D) -- (M) -- (M) -- (J) -- (J) -- (A); 
  \fill[gray!10] (A) -- (F) -- (F) -- (O) -- (O) -- (J) -- (J) -- (A); 
  \fill[gray!10] (A) -- (H) -- (H) -- (Q) -- (Q) -- (J) -- (J) -- (A); 
  
  \draw[dashed] (A) -- (B); \draw[dashed] (A) -- (D);  \draw[dashed] (A) -- (F);  \draw[dashed] (A) -- (H);
  \draw[dashed] (B) -- (C); \draw[dashed] (C) -- (D); \draw[dashed] (D) -- (E); \draw[dashed] (E) -- (F); \draw (F) -- (G); \draw (G) -- (H); \draw (H) -- (I);
  \draw (I) -- (B);
  
  \draw (J) -- (K); \draw (J) -- (M);  \draw (J) -- (O);  \draw (J) -- (Q);
  \draw (K) -- (L); \draw (L) -- (M); \draw (M) -- (N); \draw (N) -- (O); \draw (O) -- (P); \draw (P) -- (Q); \draw (Q) -- (R);
  \draw (R) -- (K);
  
  \draw[dashed] (A) -- (J); \draw (B) -- (K); \draw[dashed] (C) -- (L); \draw[dashed] (D) -- (M); \draw[dashed] (E) -- (N);
  \draw (F) -- (O); \draw (G) -- (P); \draw (H) -- (Q); \draw (I) -- (R);

  \fill (A) circle (1pt); \fill (B) circle (1pt); \fill (C) circle (1pt); \fill (D) circle (1pt);
  \fill (E) circle (1pt); \fill (F) circle (1pt); \fill (G) circle (1pt); \fill (H) circle (1pt); 
  \fill (I) circle (1pt);  \fill (J) circle (1pt); \fill (K) circle (1pt); \fill (L) circle (1pt);
  \fill (M) circle (1pt); \fill (N) circle (1pt); \fill (O) circle (1pt); \fill (P) circle (1pt);
  \fill (Q) circle (1pt); \fill (R) circle (1pt);  

  \draw[->] (-0.1,0.25) -- (-0.11,0.9);
  \draw[->] (0.15,0.1) -- (0.9,0.05);
  \draw[->] (0.01,0.2) -- (0.17,0.7);
  \draw[->] (0.15,0.15) -- (0.33,0.65);
  \draw[->] (0.15,-0.1) -- (0.9,-0.15);
  \draw[->] (0.15,-0.15) -- (0.3,-0.7);
  \draw[->] (-0.1,-0.15) -- (0.05,-0.7);
  \draw[->] (-0.1,-0.1) -- (-0.8,-0.2);
  \draw[->] (-0.1,0.1) -- (-0.8,0.0);

  \node at (0.78,0.23) {\tiny $\xi_1$};
  \node at (0.4,0.33) {\tiny $\xi_2$};
  \node at (0.78,-0.3) {\tiny $\xi_2$};
  \node at (0.4,-0.4) {\tiny $\xi_1$};
  \node at (-0.14,-0.7) {\tiny $\xi_2$};
  \node at (-0.9,-0.3) {\tiny $\xi_1$};
  \node at (-0.65,0.18) {\tiny $\xi_2$};
  \node at (-0.28,0.8) {\tiny $\xi_3$};
  \node at (0.2,0.8) {\tiny $\xi_1$};
  \node at (-0.8,-0.8) {\scriptsize $\patch^{(2)}$};
  \node at (1.3,-0.5) {\scriptsize $\patch^{(3)}$};
  \node at (1.0,2.9) {\scriptsize $\patch^{(0)}$};
  \node at (-0.5,2.7) {\scriptsize $\patch^{(1)}$};
  \node at (-1.32,1.6) {\scriptsize $\face^{(1)}$};
  \node at (0.5,2.5) {\scriptsize $\face^{(0)}$};
  \node at (-0.3,1.1) {\scriptsize $\edge^{(0)}$};
  \node at (0.3,1.2) {\scriptsize $\face^{(2)}$};
  \node at (1.3,1.9) {\scriptsize $\face^{(3)}$};
  \node at (-0.35,-0.4) {\scriptsize $\vertex^{(0)}$};
  \node at (2.5,0) {\scriptsize $\vertex^{(1)}$};
  \node at (1.4,1.5) {\scriptsize $\vertex^{(2)}$};
  \node at (0.85,1.6) {\scriptsize $\vertex^{(3)}$};
  \node at (-0.7,1.5) {\scriptsize $\vertex^{(4)}$};
  \node at (-2.05,-0.2) {\scriptsize $\vertex^{(5)}$};
  \node at (-1.4,-1.55) {\scriptsize $\vertex^{(6)}$};
  \node at (0.3,-1.95) {\scriptsize $\vertex^{(7)}$};
  \node at (2.0,-1.65) {\scriptsize $\vertex^{(8)}$};
  \node at (-0.2,2.35) {\scriptsize $\vertex^{(9)}$};
  \node at (2.4,2.2) {\scriptsize $\vertex^{(10)}$};
  \node at (2.0,3.5) {\scriptsize $\vertex^{(11)}$};
  \node at (0.5,3.8) {\scriptsize $\vertex^{(12)}$};
  \node at (-1.3,3.5) {\scriptsize $\vertex^{(13)}$};
  \node at (-2.1,2.0) {\scriptsize $\vertex^{(14)}$};
  \node at (-0.9,0.5) {\scriptsize $\vertex^{(15)}$};
  \node at (0.75,0.75) {\scriptsize $\vertex^{(16)}$};
  \node at (1.35,0.45) {\scriptsize $\vertex^{(17)}$};
  \end{tikzpicture}
 }
 \end{tabular}
\end{center}
\caption{A trilinear three- and four-patch volume~$\overline{\domain}=\cup_{i=0}^{\numPatches-1} \overline{\patch^{(i)}} \in \classVolume$ 
(i.e. $\numPatches =3$ and $\numPatches =4$, respectively) with $\nu$ inner faces~$\face^{(i)}$, $i \in \{0,1,\ldots, \numPatches-1 \}$, and one 
inner edge~$\edge^{(0)}$. The associated geometry mappings~$\map^{(i)}$, $i \in \{0,1,\ldots,\numPatches-1 \}$, are assumed to be in standard form with 
respect to the inner faces~$\face^{(i)}$, $i \in \{0,1,\ldots, \numPatches-1\}$, and with respect to the inner edge~$\edge^{(0)}$.}
\label{fig:MultiPatchVolume_InnerEdge}
\end{figure}

Based on the general framework for the construction of the $C^1$ isogeometric spline space~$\V^1$ presented in Section~\ref{subsec:general_framework}, 
and following the ideas from Remark~\ref{rem:modified_construction} for a slightly modified design, the space~$\V^1$ can be generated for a 
trilinear multi-patch volume~$\overline{\domain}$ belonging to the particular subclass $\classVolume$ of multi-patch volumes as the direct sum
\begin{equation} \label{eq:direct_sum_example}
 \V^1= \left( \bigoplus_{i=0}^{\numPatches-1} \SpacePatchMod{i} \right) \oplus 
 \left( \bigoplus_{i=0}^{\numPatches-1} \SpaceFaceMod{i} \right) \oplus \SpaceEdgeAllMod,
\end{equation}
with the single subspaces~$\SpacePatchMod{i}$, $i \in \{0,1,\ldots,\numPatches-1 \}$, $\SpaceFaceMod{i}$, $i \in \{0,1,\ldots,\numPatches-1 \}$, and 
$\SpaceEdgeAllMod$ given as
\begin{equation} \label{eq:patch_space_ex}
 \SpacePatchMod{i} = \Span \left\{ \phiPatch{i}{j_1}{j_2}{j_3}  \;|\; j_1,j_2 \in \I \setminus \{0,1 \}, \, j_3 \in \I \right\},
\end{equation}
\begin{equation} \label{eq:face_space_ex}
 \SpaceFaceMod{i} = \Span \left\{  \phiFace{i}{j_1}{j_2}{j_3}  \;|\; j_1=0,1, \, j_2,j_3 \in \I_{j_1} \setminus 
 \{0,1,\ldots,2-j_1\} \right\},
 \end{equation}
and
\begin{equation*}
 \SpaceEdgeAllMod = \left\{ \Big( \sum_{i =0}^{\numPatches-1} \phiFaceAllResMod{i} \Big) - \phiOneEdgeAll{0}  
 \; | \; \coeffAllEdgeMod \in \R^{|\coeffAllEdgeMod|}, \; \TMatrixEdgeMod \coeffAllEdgeMod = \f{0} \right\},
\end{equation*}
respectively, where $\phiFaceAllResMod{i}$ is the function
\[
 \phiFaceAllResMod{i}(\f{x})= \sum_{j_1=0}^{1} \sum_{j_2=0}^{2-j_1} \sum_{j_3=0}^{2-j_1} \coeffFace{i}{j_1}{j_2}{j_3} 
 \phiFace{i}{j_1}{j_2}{j_3}(\f{x}),
\]
$\coeffAllEdgeMod$ is the vector of all coefficients of the functions~$\,\phiFaceAllResMod{i}$, $i \in \{0,1,\ldots, \numPatches -1 \}$, and 
$\phiOneEdgeAll{0}$, i.e. $\coeffFace{i}{j_1}{j_2}{j_3}$, $j_1=0,1$, $ j_2,j_3 \in  \{0,1,\ldots,2-j_1\}$, $i \in \{
0,1,\ldots,\numPatches\}$, and $\coeffPatchEdgeSt{0}{j_1}{j_2}{j_3}{m}$, $j_1,j_2=0,1$, $j_3 \in \I$, $m \in \{0,1,\ldots,\numPatches -1 \}$, and 
\begin{equation} \label{eq:homogoneous_system_Mod} 
 \TMatrixEdgeMod \coeffAllEdgeMod = \f{0},
\end{equation}
is the reduced homogeneous linear system~\eqref{eq:homogoneous_system} formed by the remaining equations~\eqref{eq:facesequalNew} and 
\eqref{eq:facesedgesequalNew}, which are given by the equations
\begin{equation*}
 \partial_{1}^{\ell_1} \partial_{2}^{\ell_2} \big(\phiFaceAll{\LeftFace{0}{m}} \circ \map^{(m)} \big)(0,0,\zeta_j) = 
 \partial_{1}^{\ell_1} \partial_{2}^{\ell_2} \big(\phiFaceAll{\RightFace{0}{m}} \circ \map^{(m)} \big)(0,0,\zeta_j), 
 \mbox{ } 0 \leq \ell_1,\ell_2, \leq 1, \, j \in \I,
\end{equation*}
for $m \in \{0,1,\ldots,\numPatches-1 \}$, and
\begin{equation*} 
  \partial_{1}^{\ell_1} \partial_{2}^{\ell_2} \big(\phiFaceAll{\LeftFace{0}{m}} \circ \map^{(m)} \big)(0,0,\zeta_j) = 
 \partial_{1}^{\ell_1} \partial_{2}^{\ell_2} \big(\phiOneEdgeAll{0} \circ \map^{(m)} \big)(0,0,\zeta_j), 
 \mbox{ } 0 \leq \ell_1,\ell_2, \leq 1, \, j \in \I,
\end{equation*}
for $m \in \{0,1,\ldots,\numPatches\}$. Note that the here presented construction~\eqref{eq:direct_sum_example} of the $C^1$ isogeometric spline 
space~$\V^1$ for the particular subclass $\classVolume$ of trilinear multi-patch volumes~$\overline{\domain}$ can be directly derived from the general 
framework~\eqref{eq:direct_sum_GeneralMultiPatch}. Now, instead of using the patch spaces~$\SpacePatch{i}$, $i \in \indexPatch$, we can 
trivially extend these spaces to the boundary faces and to the boundary edges with a patch valency~$1$ of the multi-patch volume~$\overline{\domain}$ to 
get the modified patch spaces~$\SpacePatchMod{i}$. Therefore, the face spaces~$\SpaceFace{i}$ for the boundary faces~$\face^{(i)}$, $i \in 
\indexFaceBoundary$, are not needed anymore, since each face space~$\SpaceFace{i}$, $i \in \indexFaceBoundary$, is contained in a patch 
space~$\SpacePatchMod{\ell}$, $\ell \in \indexPatch$. Furthermore, instead of using the face spaces~$\SpaceFace{i}$ for the inner faces~$\face^{(i)}$, 
$i \in \indexFaceInner$, these spaces can be trivially enlarged to the boundary edges with a patch valency~$2$ of the multi-patch 
volume~$\overline{\domain}$ to obtain the adapted face spaces~$\SpaceFaceMod{i}$ for $i \in \indexFaceInner$. As a result of these modifications, the 
edge space~$\SpaceEdgeAll$ reduces to the smaller and simplified edge space~$\SpaceEdgeAllMod$, where now just one edge, namely the inner 
edge~$\edge^{(0)}$, has to be considered.

Analogous to space~$\SpaceEdgeAll$ in Section~\ref{subsec:general_framework}, a basis of the space~$\SpaceEdgeAllMod$ is determined by 
a basis of the kernel of the matrix~$\TMatrixEdgeMod$ in the homogeneous linear system~\eqref{eq:homogoneous_system_Mod}, and can be constructed 
again e.g. by finding a minimal determining set for the coefficients~$\coeffAllEdgeMod$. Let such a basis of the space~$\SpaceEdgeAllMod$ be given 
by the functions~$\phiEdgeAllSingleMod{j}$, $j \in \{0,1,\ldots,n_{\TMatrixEdgeMod}-1 \}$, with $n_{\TMatrixEdgeMod} = \dim \ker (\TMatrixEdgeMod) = 
\dim \SpaceEdgeAllMod$, i.e.
\[
 \SpaceEdgeAllMod = \Span \left\{ \phiEdgeAllSingleMod{j} \;|\; j \in \{0,1,\ldots,n_{\TMatrixEdgeMod} -1 \} \right\},
\]
then the functions $\phiEdgeAllSingleMod{j}$, $j \in \{0,1,\ldots,n_{\TMatrixEdgeMod}-1 \}$, form together with the functions 
$\phiPatch{i}{j_1}{j_2}{j_3}$, $j_1,j_2 \in \I \setminus \{0,1 \}$, $j_3 \in \I$, $i \in \{0,1,\ldots,\numPatches -1 \}$, and 
$\phiFace{i}{j_1}{j_2}{j_3}$, $j_1=0,1$, $j_2,j_3 \in \I_{j_1} \setminus \{0,1,\ldots,2-j_1\}$, $i \in \{0,1,\ldots,\numPatches -1 \}$, a basis of 
the $C^1$ isogeometric spline space~$\V^1$. 

%%%%%%%%%%%%%%%%%%%%%%%%%%%%%%%%%%%%%%%%%%%%%%%%%%%%%%%%%%%%%%%%%%%%%%%%
\subsection{Dimension of $\V^1$} \label{subsec:dimension}
%%%%%%%%%%%%%%%%%%%%%%%%%%%%%%%%%%%%%%%%%%%%%%%%%%%%%%%%%%%%%%%%%%%%%%%%

Due to the possible decomposition of the $C^1$ isogeometric spline space~$\V^1$ into the direct sum~\eqref{eq:direct_sum_example}, the dimension of $\V^1$ 
can be obtained via
\[
  \dim \V^1 = \sum_{i =0}^{\numPatches-1} \dim \SpacePatchMod{i}  +  \sum_{i =0}^{\numPatches -1} \dim \SpaceFaceMod{i}  + \dim \SpaceEdgeAllMod .
\]
The dimensions of the spaces $\SpacePatchMod{i}$, $i \in \{0,1,\ldots, \nu-1 \}$, and of the spaces~$\SpaceFaceMod{i}$, $i \in \{0,1,\ldots, \nu-1 \}$, are 
equal to
\[
 \dim \SpacePatchMod{i} = |\I|(|\I|-2)^2, \quad i \in \{0,1,\ldots, \numPatches-1 \}, 
\]
and 
\[
 \dim \SpaceFaceMod{i} = (|\I_0|-6)^2 + (|\I_1|-4)^2 , \quad i \in \{0,1,\ldots, \numPatches-1 \}, 
\]
respectively, which directly follow from the constructions~\eqref{eq:patch_space_ex} or \eqref{eq:face_space_ex} for the single spaces. 
As already 
explained in Section~\ref{subsec:general_framework}, the dimension of~$\SpaceEdgeAllMod$ does not just depend on the degree~$p$, the regularity~$r$ and the 
number of spline elements, i.e. $k+1$, of the underlying spline space~$\US{p}{r}{h}$ like for the dimensions of the spaces $\SpacePatchMod{i}$ and 
$\SpaceFaceMod{i}$, rather also on the valency~$\nu$ of the inner edge and even on the shapes of the individual trilinear patches of the multi-patch 
domain, which will be also seen later on the basis of an example. We numerically compute the generic dimension of the space~$\SpaceEdgeAllMod$, that 
is the dimension of the space~$\SpaceEdgeAllMod$ which can be expected to get with probability $1$ for a trilinear multi-patch 
volume~$\overline{\domain} \in \classVolume$ with an inner edge of valency~$\nu$ and for a given spline degree $p$, regularity $r$ and $k+1$ 
spline elements.
For this numerical study, we identify for a large number of randomly generated trilinear multi-patch volumes 
$\overline{\domain} \in \classVolume$ with an inner edge valency~$\nu=3$, $\nu=4$ or $\nu=5$  the dimension of the space~$\SpaceEdgeAllMod$ for different 
values of~$p$, $r$ and $k$, see Table~\ref{tab:examples_inner_edge}, and we conjecture from these results that 
the generic dimension of the spline space~$\SpaceEdgeAllMod$ is given as
\begin{equation} \label{eq:generic_dimension}
\dim \SpaceEdgeAllMod = 3p+1 + \numPatches (p-1) + k \max \left(0,(\numPatches + 3) (p-r-3) + 3 \right) .
\end{equation}
Therefore, the space~$\SpaceEdgeAllMod$ is h-refineable in the generic case if $  p-r > \lfloor \frac{3 (\numPatches + 2)}{\numPatches +3}\rfloor = 2 $.
{\renewcommand{\arraystretch}{1.6}
\begin{table}[htb]
 \centering\scriptsize 
 \begin{tabular}{|c||c|c|c|c|c|c|c|c|c|c|c|c|c|} \hline
 \multicolumn{13}{|c|}{$\nu=3$} \\ \hline
 & $p=3$ & \multicolumn{2}{|c|}{$p=4$} & \multicolumn{3}{|c|}{$p=5$} & \multicolumn{3}{|c|}{$p=6$} & \multicolumn{3}{|c|}{$p=7$} \\ \hline  
 $k$ & $r=1$ & $r=1$ & $r=2$ & $r=1$ & $r=2$ & $r=3$ & $r=1$ & $r=2$ & $r=3$ & $r=1$ & $r=2$ & $r=3$ \\ \hline 
 0 & 16 & 22 & 22 & 28 & 28 & 28 & 34 & 34 & 34 & 40  & 40  & 40 \\ \hline
 1 & 16 & 25 & 22 & 37 & 31 & 28 & 49 & 43 & 37 & 61  & 55  & 49 \\ \hline
 2 & 16 & 28 & 22 & 46 & 34 & 28 & 64 & 52 & 40 & 82  & 70  & 58 \\ \hline
 3 & 16 & 31 & 22 & 55 & 37 & 28 & 79 & 61 & 43 & 103 & 85  & 67 \\ \hline
 4 & 16 & 34 & 22 & 64 & 40 & 28 & 94 & 70 & 46 & 124 & 100 & 76 \\ \hline \hline
  \multicolumn{13}{|c|}{$\nu=4$} \\ \hline
 & $p=3$ & \multicolumn{2}{|c|}{$p=4$} & \multicolumn{3}{|c|}{$p=5$} & \multicolumn{3}{|c|}{$p=6$} & \multicolumn{3}{|c|}{$p=7$} \\ \hline  
 $k$ & $r=1$ & $r=1$ & $r=2$ & $r=1$ & $r=2$ & $r=3$ & $r=1$ & $r=2$ & $r=3$ & $r=1$ & $r=2$ & $r=3$ \\ \hline 
 0 & 18 & 25 & 25 & 32 & 32 & 32 & 39  & 39 & 39 & 46  & 46  & 46 \\ \hline
 1 & 18 & 28 & 25 & 42 & 35 & 32 & 56  & 49 & 42 & 70  & 63  & 56 \\ \hline
 2 & 18 & 31 & 25 & 52 & 38 & 32 & 73  & 59 & 45 & 94  & 80  & 66 \\ \hline
 3 & 18 & 34 & 25 & 62 & 41 & 32 & 90  & 69 & 48 & 118 & 97  & 76 \\ \hline
 4 & 18 & 37 & 25 & 72 & 44 & 32 & 107 & 79 & 51 & 142 & 114 & 86 \\ \hline \hline
  \multicolumn{13}{|c|}{$\nu=5$} \\ \hline
 & $p=3$ & \multicolumn{2}{|c|}{$p=4$} & \multicolumn{3}{|c|}{$p=5$} & \multicolumn{3}{|c|}{$p=6$} & \multicolumn{3}{|c|}{$p=7$} \\ \hline  
 $k$ & $r=1$ & $r=1$ & $r=2$ & $r=1$ & $r=2$ & $r=3$ & $r=1$ & $r=2$ & $r=3$ & $r=1$ & $r=2$ & $r=3$ \\ \hline 
 0 & 20 & 28 & 28 & 36 & 36 & 36 & 44  & 44 & 44 & 52  & 52  & 52 \\ \hline
 1 & 20 & 31 & 28 & 47 & 39 & 36 & 63  & 55 & 47 & 79  & 71  & 63 \\ \hline
 2 & 20 & 34 & 28 & 58 & 42 & 36 & 82  & 66 & 50 & 106 & 90  & 74\\ \hline
 3 & 20 & 37 & 28 & 69 & 45 & 36 & 101  & 77 & 53 & 133 & 109 & 85\\ \hline
 4 & 20 & 40 & 28 & 80 & 48 & 36 & 120 & 88 & 56 & 160 & 128 & 96\\ \hline 
  \end{tabular}
  \caption{Numerically obtained generic dimension of the space $\SpaceEdgeAllMod$ for different inner edge valency~$\nu$, spline degree~$p$, regularity $r$ 
  and $k+1$ spline elements.}
  \label{tab:examples_inner_edge}
\end{table}
}

 For some multi-patch volumes~$\overline{\domain} \in \classVolume$, in particular with special configurations and shapes of the single 
trilinear patches, the dimension of the corresponding space~$\SpaceEdgeAllMod$ can differ from the generic dimension~\eqref{eq:generic_dimension}. 
One such example is a four-patch domain~$\overline{\domain} \in \classVolume$ as shown in Fig.~\ref{fig:MultiPatchVolume_InnerEdge}~(right) with vertices
$\vertex^{(i)}$, $i \in \{0, 1, \ldots,17 \}$, which are given as
\[
 \vertex^{(0)} = \left(\frac{13}{16},\frac{7}{20},-\frac{21}{80} \right)^T, \mbox{ }
 \vertex^{(1)} = \left(\frac{83}{20},\frac{1}{5},-\frac{1}{5} \right)^T, \mbox{ }
 \vertex^{(2)} = \left(\frac{9}{2},\frac{43}{10},-\frac{3}{20} \right)^T, \]
 \[\vertex^{(3)} = \left(\frac{3}{10},\frac{22}{5},-\frac{11}{40} \right)^T, \mbox{ }
 \vertex^{(4)} = \left(-\frac{39}{10},\frac{9}{2},-\frac{2}{5} \right)^T, \mbox{ }
 \vertex^{(5)} = \left(-\frac{151}{40},\frac{1}{2},-\frac{13}{40} \right)^T, \]
 \begin{equation} \label{eq:vertices_non-generic}
 \vertex^{(6)} = \left(-\frac{73}{20},-\frac{7}{2},-\frac{1}{4} \right)^T, \mbox{ }
 \vertex^{(7)} = \left(\frac{3}{40},-\frac{37}{10},-\frac{1}{4} \right)^T, \mbox{ }
 \vertex^{(8)} = \left(\frac{19}{5},-\frac{39}{10},-\frac{1}{4} \right)^T,
 \end{equation}
 \[\vertex^{(9)} = \left(-\frac{1}{15},-\frac{13}{30},\frac{197}{30} \right)^T, \mbox{ } 
 \vertex^{(10)} = \left(\frac{43}{10},\frac{3}{5},\frac{16}{3} \right)^T, \mbox{ }
 \vertex^{(11)} = \left(\frac{19}{5},\frac{121}{30},\frac{29}{5} \right)^T,
\]
\[
 \vertex^{(12)} = \left(-\frac{4}{15},\frac{18}{5},\frac{91}{15} \right)^T, \mbox{ }
 \vertex^{(13)} = \left(-\frac{107}{30},\frac{133}{30},\frac{173}{30} \right)^T, \mbox{ }
 \vertex^{(14)} = \left(-\frac{139}{30},\frac{2}{3},\frac{191}{30} \right)^T,
\]
\[
 \vertex^{(15)} = \left(-\frac{39}{10},-\frac{53}{15},\frac{91}{15} \right)^T, \mbox{ }
 \vertex^{(16)} = \left(\frac{2}{3},-\frac{121}{30},\frac{193}{30} \right)^T, \mbox{ }
 \vertex^{(17)} = \left(\frac{7}{2},-\frac{23}{6},\frac{193}{30} \right)^T.
\]
For this concrete example, the vertices~$\vertex^{(i)}$, $i \in \{0,1,\ldots,8 \}$, are lying on the planar surface
\[
\surf(\xi_1,\xi_2)= (1-\xi_1)(1-\xi_2) \vertex^{(6)} + \xi_1(1-\xi_2) \vertex^{(8)} + (1-\xi_1)\xi_2 \vertex^{(4)} + \xi_1 \xi_2 \vertex^{(2)}, 
\mbox{ } (\xi_1,\xi_2) \in [0,1 ]^2,
\]
and are determined via this surface by 
\[
 \vertex^{(0)} = \surf\left(\frac{1}{2},\frac{1}{2}\right), \mbox{ }  \vertex^{(1)} = \surf\left(1,\frac{1}{2}\right), \mbox{ }   
 \vertex^{(2)} = \surf\left(1,1\right), \mbox{ }   \vertex^{(3)} = \surf\left(\frac{1}{2},1\right), \mbox{ } 
 \vertex^{(4)} = \surf\left(0,1\right),
\]
\[
 \vertex^{(5)} = \surf\left(0,\frac{1}{2}\right), \mbox{ }  \vertex^{(6)} = \surf(0,0), \mbox{ }   
 \vertex^{(7)} = \surf\left(\frac{1}{2},0\right), \mbox{ }   \vertex^{(8)} = \surf(1,0).
\]
As in the numerical study for the generic case above, we compute the dimension of the space~$\SpaceEdgeAllMod$ for different values of~$p$, $r$ and $k$, 
see Table~\ref{tab:example_non-generic}. The results indicate that the dimension of the space~$\SpaceEdgeAllMod$ is equal to
\begin{equation*}
\dim \SpaceEdgeAllMod = 3p+2 + \numPatches (p-1) + k \max \left(0,(\numPatches + 3) (p-r-3) + 4 \right) ,
\end{equation*}
and is hence slightly larger than the generic one~\eqref{eq:generic_dimension}.

\begin{table}[htb]
 \centering\footnotesize  
 \begin{tabular}{|c||c|c|c|c|c|c|c|c|c|c|c|c|c|} \hline
 & $p=3$ & \multicolumn{2}{|c|}{$p=4$} & \multicolumn{3}{|c|}{$p=5$} & \multicolumn{3}{|c|}{$p=6$} & \multicolumn{3}{|c|}{$p=7$} \\ \hline  
 $k$ & $r=1$ & $r=1$ & $r=2$ & $r=1$ & $r=2$ & $r=3$ & $r=1$ & $r=2$ & $r=3$ & $r=1$ & $r=2$ & $r=3$ \\ \hline 
 0 & 19 & 26 & 26 & 33 & 33 & 33 & 40  & 40 & 40 & 47  & 47   & 47 \\ \hline
 1 & 19 & 30 & 26 & 44 & 37 & 33 & 58  & 51 & 44 & 72  & 65   & 58 \\ \hline
 2 & 19 & 34 & 26 & 55 & 41 & 33 & 76  & 62 & 48 & 97  & 83   & 69 \\ \hline
 3 & 19 & 38 & 26 & 66 & 45 & 33 & 94  & 73 & 52 & 122 & 101  & 80 \\ \hline
 4 & 19 & 42 & 26 & 77 & 49 & 33 & 112 & 84 & 56 & 147 & 119  & 91 \\ \hline
  \end{tabular}
  \caption{The dimension of the space $\SpaceEdgeAllMod$ for different spline degree~$p$, regularity $r$ and $k+1$ spline elements for a 
  non-generic example of a multi-patch volume~$\overline{\domain} \in \classVolume$ with an inner edge valency~$\numPatches=4$, where the 
  vertices~$\vertex^{(i)}$, $i \in \{0,1, \ldots, 17 \}$, are given in~\eqref{eq:vertices_non-generic}.}
  \label{tab:example_non-generic}
\end{table}

%%%%%%%%%%%%%%%%%%%%%%%%%%%%%%%%%%%%%%%%%%%%%%%%%%%%%%%%%%%%%%%%%%%%%%%%
\subsection{$L^2$ approximation} \label{subsec:L2approximation}
%%%%%%%%%%%%%%%%%%%%%%%%%%%%%%%%%%%%%%%%%%%%%%%%%%%%%%%%%%%%%%%%%%%%%%%%

The goal of this subsection is to numerically explore the approximation properties of the space~$\V^1$ over a trilinear multi-patch 
volume~$\overline{\domain} \in \classVolume$. For this purpose, we perform $L^2$~approximation on one concrete volume, namely on a trilinear three-patch 
volume~$\overline{\domain} \in \classVolume$ as shown in Fig.~\ref{fig:MultiPatchVolume_InnerEdge}~(left), where the single 
vertices~$\vertex^{(i)}$, $i \in \{0,1,\ldots,13 \}$, are given as
\[
 \vertex^{(0)} = \left( \frac{88}{15},\frac{9}{5},-\frac{2}{3}\right)^T, \mbox{ } \vertex^{(1)} = \left(\frac{541}{60},\frac{113}{35},-\frac{4}{15}\right)^T, 
 \mbox{ } \vertex^{(2)} = \left(\frac{94}{15},\frac{187}{30},\frac{4}{15}\right)^T,
 \]
 \[\vertex^{(3)} = \left(\frac{13}{6},\frac{93}{20},-\frac{2}{15}\right)^T, \mbox{ }
\vertex^{(4)} = \left(-\frac{3}{5},-\frac{2}{5},\frac{2}{5}\right)^T, 
 \mbox{ } \vertex^{(5)} = \left(\frac{71}{15},-\frac{29}{30},0\right)^T,
 \]
 \[\vertex^{(6)} = \left(\frac{97}{10},-\frac{2}{5},-\frac{4}{15}\right)^T, \mbox{ } \vertex^{(7)} = \left(\frac{157}{30},\frac{17}{10},\frac{13}{2}\right)^T, \mbox{ }\vertex^{(8)} = \left(\frac{113}{12},\frac{629}{210},\frac{35}{6}\right)^T,
 \]
 \[\vertex^{(9)} = \left(\frac{91}{15},\frac{101}{15},\frac{179}{30}\right)^T, \mbox{ } \vertex^{(10)} = \left(\frac{67}{30},\frac{93}{20},\frac{98}{15}\right)^T, \mbox{ }\vertex^{(11)} = \left(\frac{1}{6},\frac{17}{30},\frac{31}{5}\right)^T,
 \]
 \[\vertex^{(12)} = \left(\frac{83}{15},-\frac{6}{5},\frac{92}{15}\right)^T, \mbox{ } \vertex^{(13)} = \left(\frac{101}{10},0,\frac{59}{10}\right)^T.
\]
We generate for the spline degrees~$p=3,4,5,6$ nested $C^1$ isogometric spline spaces~$\V^1$ with regularity $r=1$ for mesh sizes 
$h=\frac{1}{2^L}$, $L=0,1,2,3,4$ for $p \in \{3,4 \}$ and $L=0,1,2,3$ for $p \in \{5,6 \}$, where $L$ is the level of refinement. The dimensions of the resulting spaces~$\V^1$ and of the corresponding subspaces $\oplus_{i=0}^2 \SpacePatchMod{i}$, $\oplus_{i=0}^2 \SpaceFaceMod{i}$, $\SpaceEdgeAllMod$ are given in Table~\ref{tab:example_L2approximation}. Let $\{ \phi_i \}_{i=0}^{\dim \V^1 -1}$ be the constructed basis of such a 
$C^1$ space~$\V^1$, then we aim at approximating the function
\begin{equation} \label{eq:approx_func}
 z(x_1,x_2,x_3) =  5 \cos \left(\frac{x_1}{2}\right) \sin \left(\frac{x_2}{2}\right) \cos \left(\frac{x_3}{2}\right)
\end{equation}
over the considered trilinear three-patch volume in a least-squares sense, that means, we compute an approximation
\[
 z_h(\f{x}) = \sum_{i=0}^{\dim \V^1 -1} c_i \, \phi_i (\f{x}) , \quad c_i \in \R,
\]
of the function~$z$, which minimizes the objective function
\[
 \int_{\Omega} (z_h(\f{x})-z(\f{x}))^2 \mathrm{d}\f{x}.
\]
Fig.~\ref{fig:examples} shows the resulting relative $L^2$ errors measured on the entire volume~$\overline{\domain}$ (left), on the union of the 
inner faces~$\cup_{i=0}^2 \face^{(i)}$ (middle) and on the inner edge~$\edge^{(0)}$ (right). While for the spline degrees~$p=3,4$ the convergence rates are quite low, which is amongst others a consequence of the constant (for $p=3$) or very slowly increasing (for $p=4$) dimension of the edge space~$\SpaceEdgeAllMod$, the convergence rates for the spline degrees~$p=5,6$ are high and illustrate good approximation properties of the corresponding $C^1$ spaces~$\V^1$.

\begin{table}[htb]
 \centering\footnotesize
 \begin{tabular}{|c||c|c|c|c|c|c|c|c|} \hline
 & \multicolumn{4}{|c|}{$p=3$} & \multicolumn{4}{|c|}{$p=4$} \\ \hline
 $L$ & $\dim \V^1 $ & $\sum_{i} \dim \SpacePatchMod{i}$ & $\sum_{i} \dim \SpaceFaceMod{i}$ & $\dim \SpaceEdgeAllMod$ & $\dim \V^1 $ & $\sum_{i} \dim \SpacePatchMod{i}$ & $\sum_{i} \dim \SpaceFaceMod{i}$ & $\dim \SpaceEdgeAllMod$ \\ \hline 
 0 & 76 & 48 & 12 & 16 & 196 & 135 & 39 & 22 \\ \hline
 1 & 334 & 288 & 30 & 16 & 997 & 864 & 108 & 25 \\ \hline
 2 & 2020 & 1920 & 84 & 16 & 6415 & 6048 & 336 & 31\\ \hline
 3 & 14104 & 13824 & 264 & 16 & 46123 & 44928 & 1152 & 43\\ \hline
 4 & 105376 & 104448 & 912 & 16 & 349891 & 345600 & 4224 & 67 \\ \hline \hline
 & \multicolumn{4}{|c|}{$p=5$} & \multicolumn{4}{|c|}{$p=6$} \\ \hline
 $L$ & $\dim \V^1 $ & $\sum_{i} \dim \SpacePatchMod{i}$ & $\sum_{i} \dim \SpaceFaceMod{i}$ & $\dim \SpaceEdgeAllMod$ & $\dim \V^1 $ & $\sum_{i} \dim \SpacePatchMod{i}$ & $\sum_{i} \dim \SpaceFaceMod{i}$ & $\dim \SpaceEdgeAllMod$ \\ \hline 
 0 & 394 & 288 &  78 & 28 & 688 & 525 & 129 & 34 \\ \hline
 1 & 2191 & 1920 & 234 & 37 & 4057 & 3600 & 408 & 49 \\ \hline
 2 & 14659 & 13824 & 780 & 55 & 27895 & 26400 & 1416 & 79 \\ \hline
 3 & 107347 & 104448 & 2808 & 91 & 206971 & 201600 & 5232 & 139\\ \hline
   \end{tabular}
  \caption{The dimension of the spaces~$\V^1$ and of the corresponding subspaces $\oplus_{i=0}^2 \SpacePatchMod{i}$, $\oplus_{i=0}^2 \SpaceFaceMod{i}$, $\SpaceEdgeAllMod$ for different degrees~$p$ and level of refinements~$L$ used in the $L^2$ approximation example in Section~\ref{subsec:L2approximation}.}
  \label{tab:example_L2approximation}
\end{table}
 
\begin{figure}[htb]
\begin{center}
\begin{tabular}{ccc}
\multicolumn{3}{c}{\scriptsize $p=3$, $r=1$ \& $p=4$, $r=1$ } \\
\includegraphics[width=4.9cm,clip]{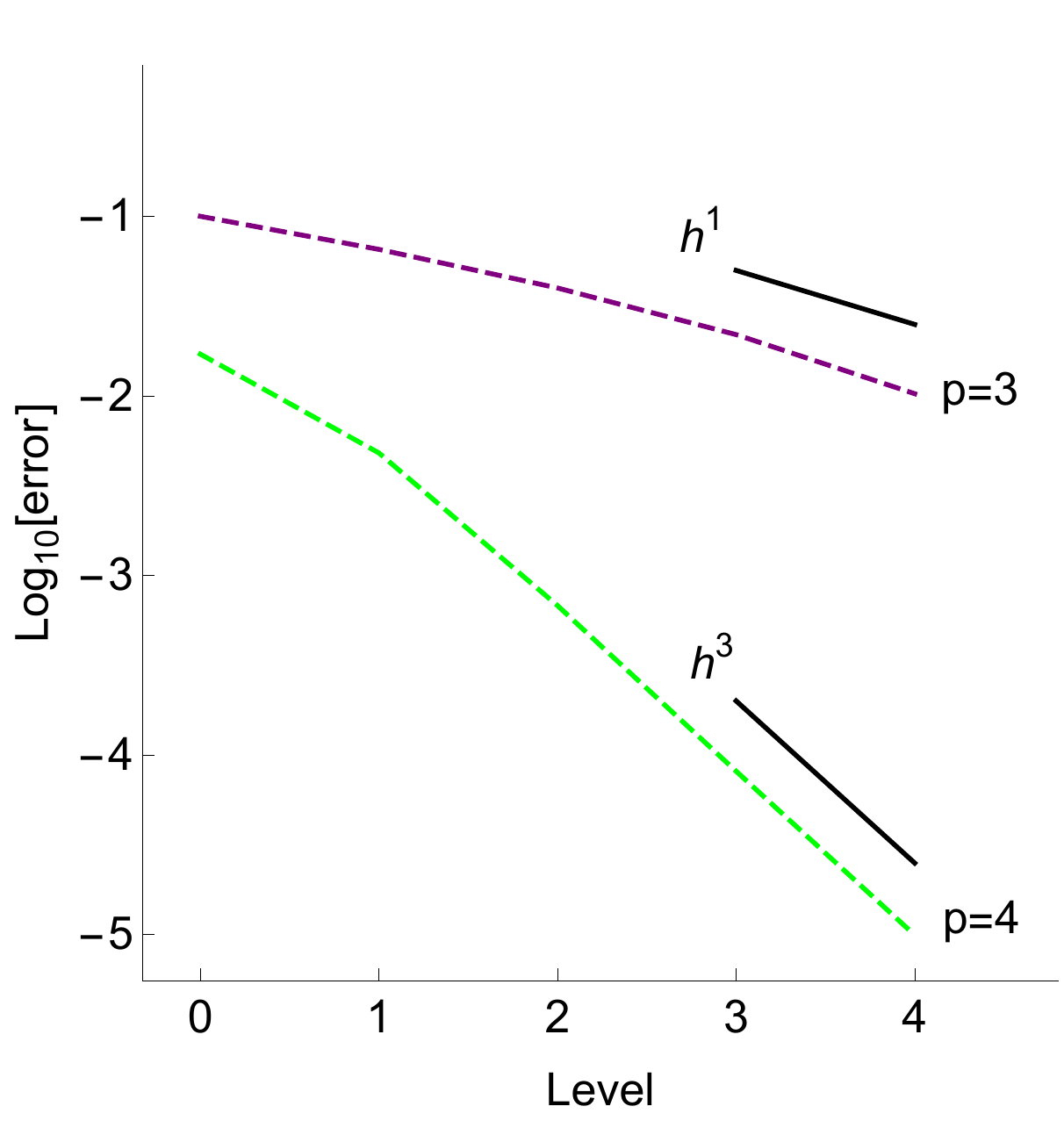} & 
\includegraphics[width=4.9cm,clip]{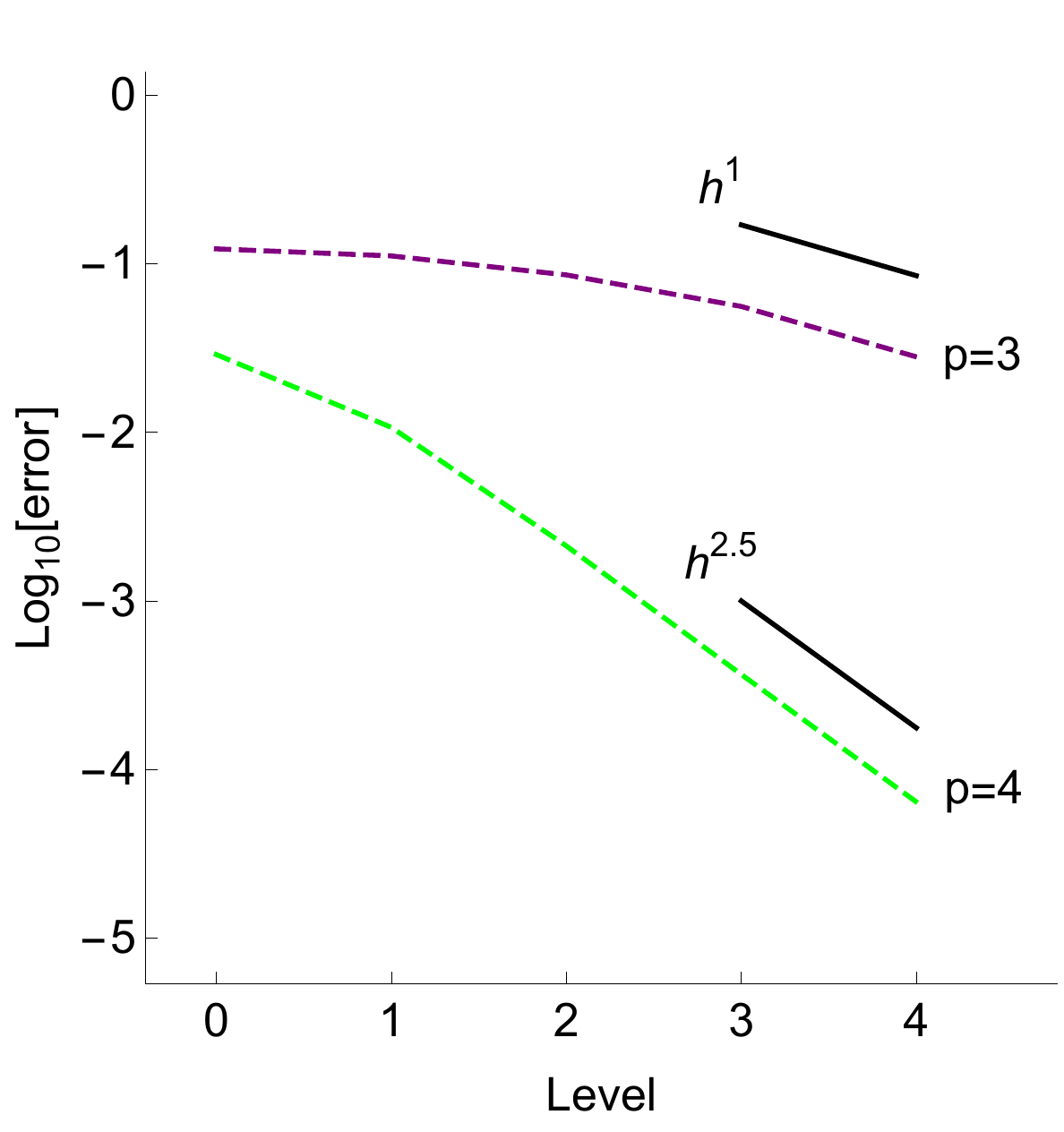} &
\includegraphics[width=4.9cm,clip]{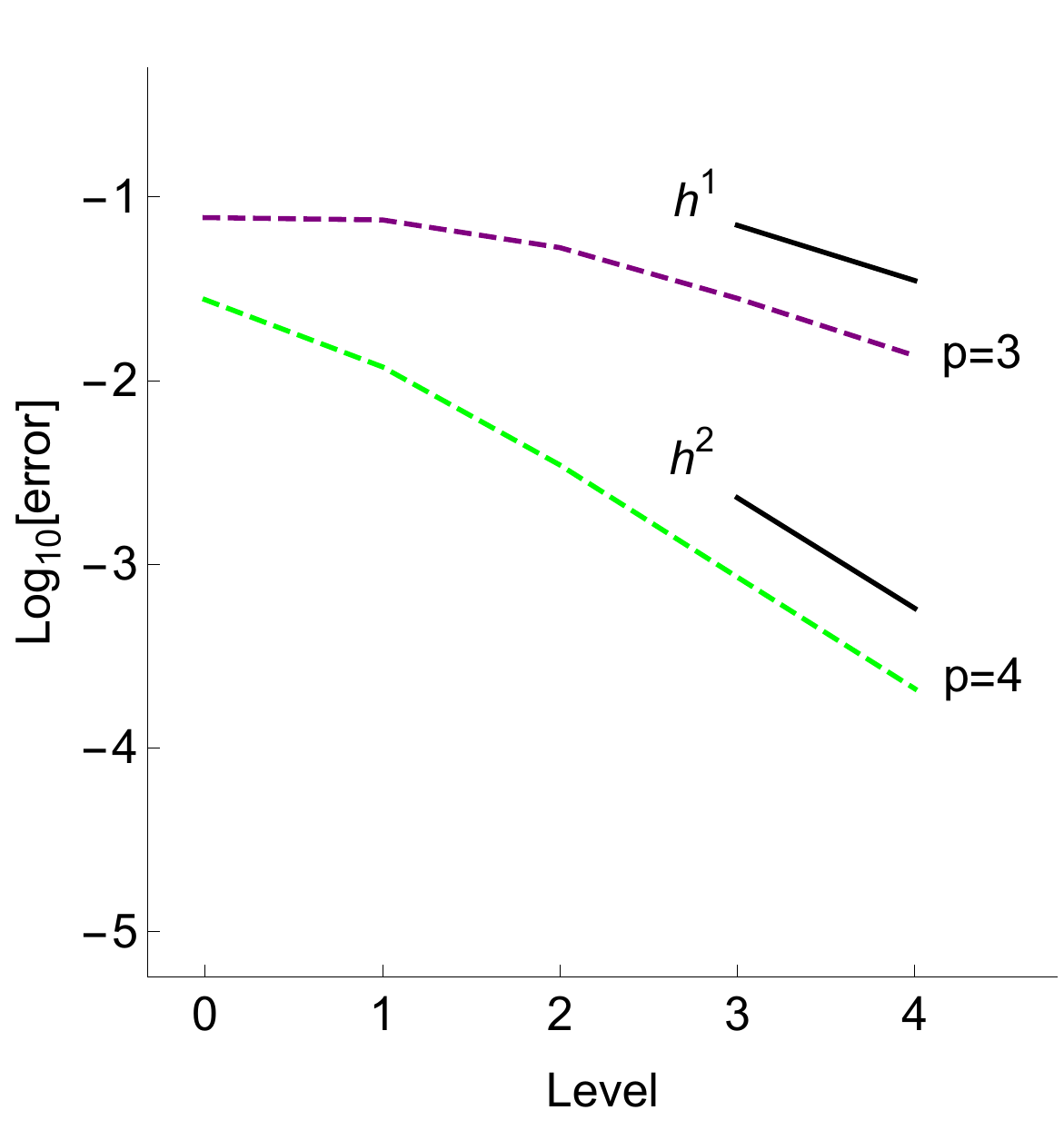} \\
\multicolumn{3}{c}{\scriptsize $p=5$, $r=1$ \& $p=6$, $r=1$ } \\
\includegraphics[width=4.9cm,clip]{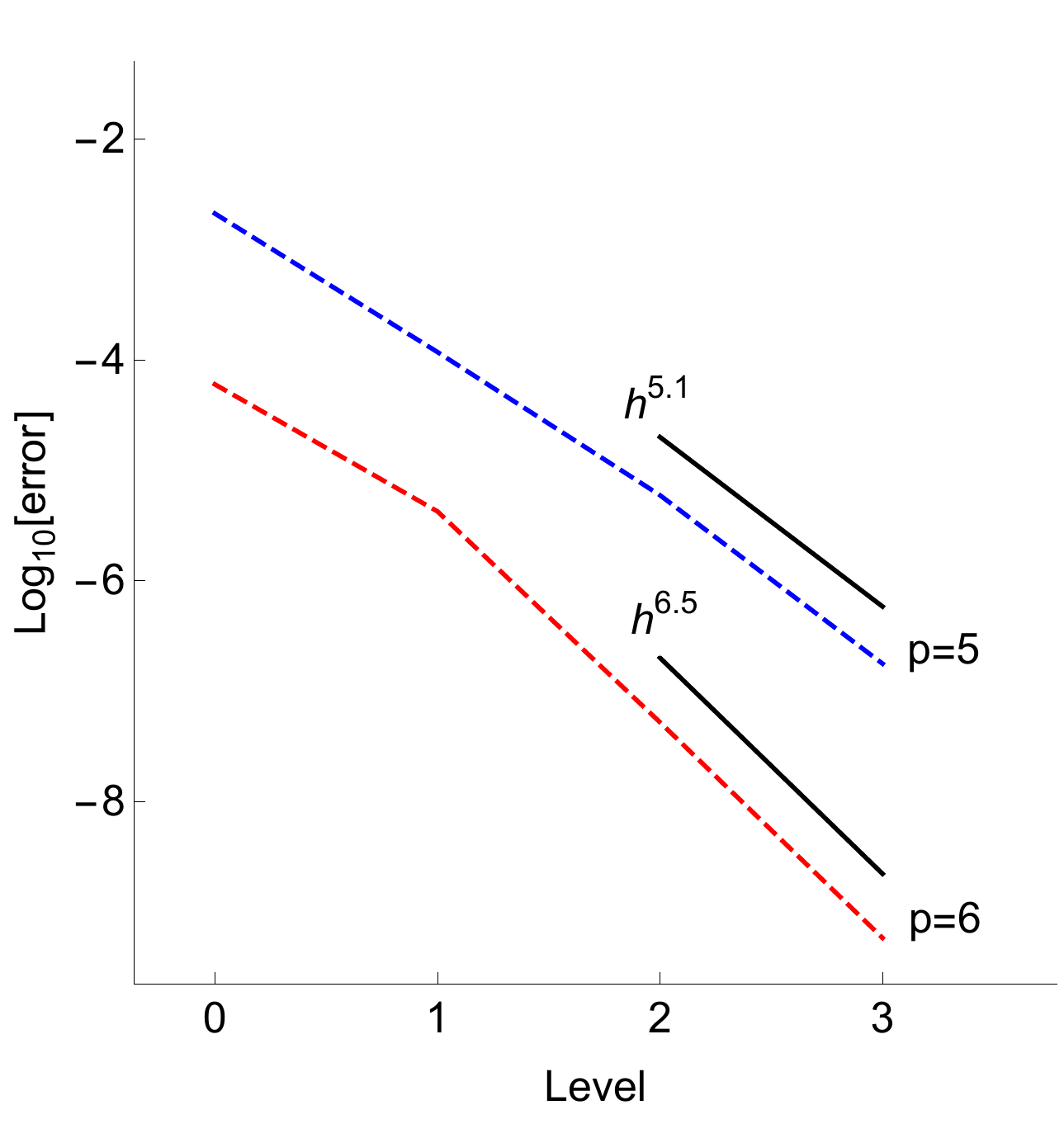} & 
\includegraphics[width=4.9cm,clip]{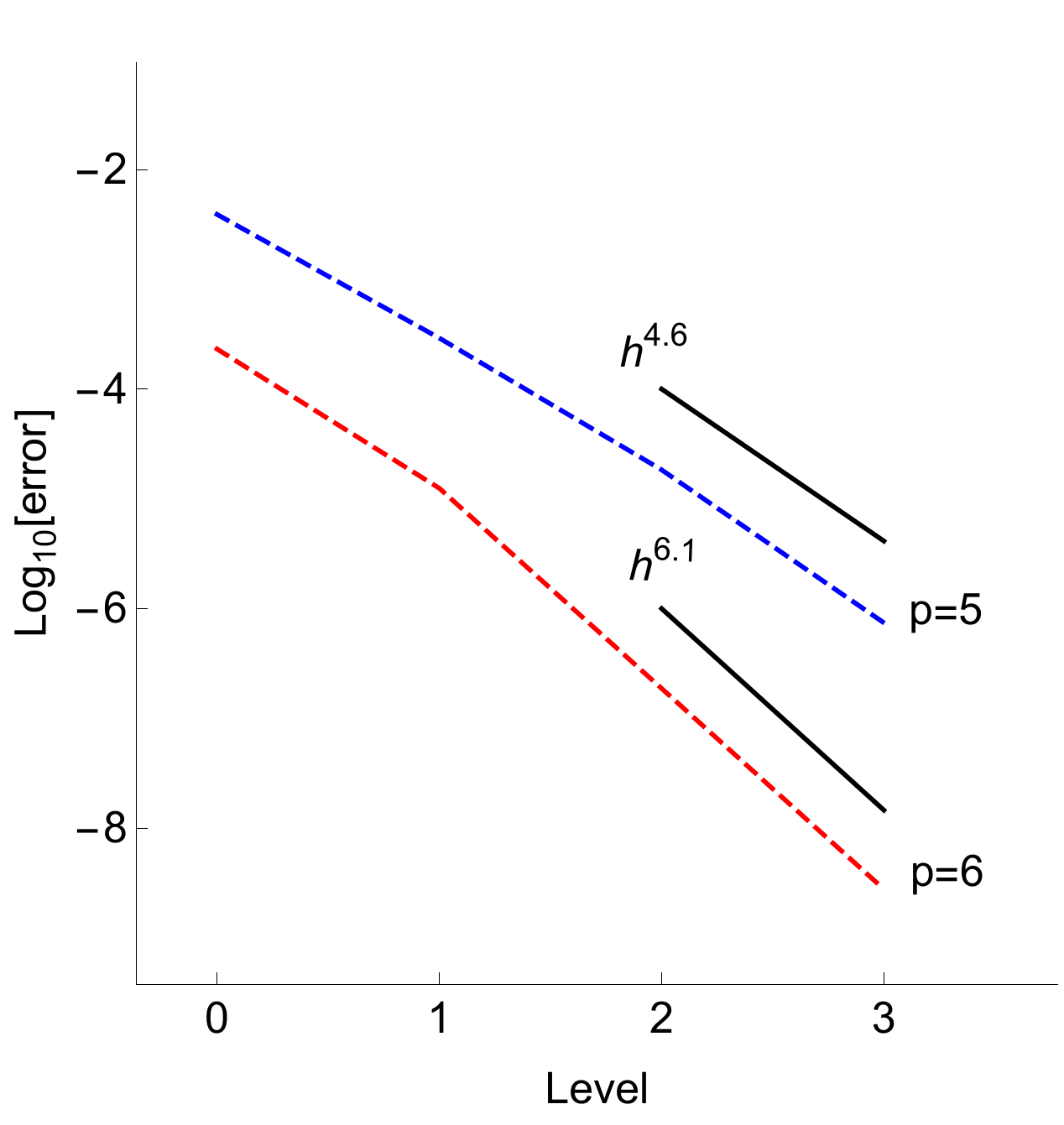} &
\includegraphics[width=4.9cm,clip]{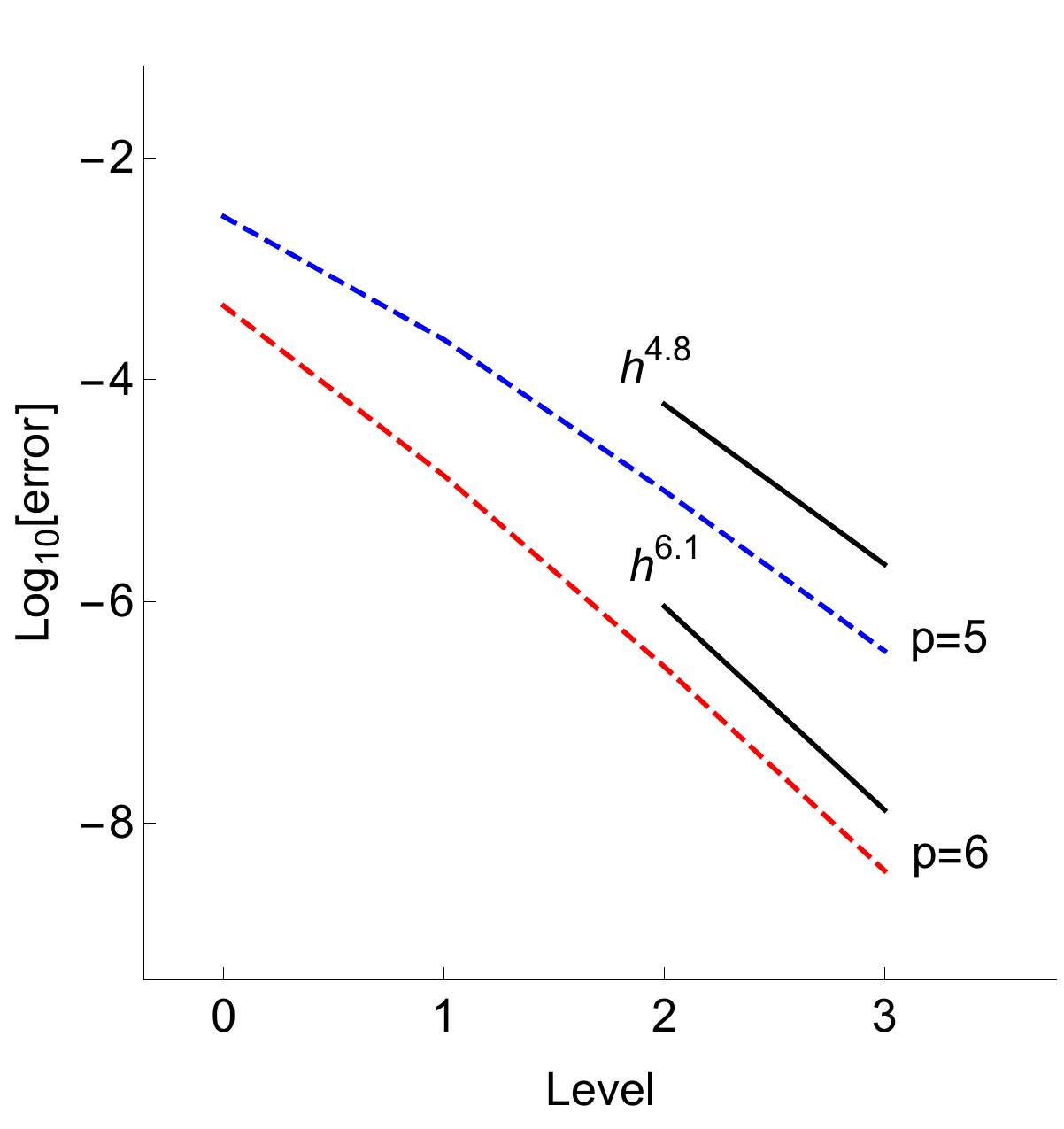} \\
{\scriptsize Rel. $L^2$-errors on entire volume~$\overline{\domain}$ } &
{\scriptsize Rel. $L^2$-errors on inner faces $\cup_{i=0}^2 \face^{(i)}$ } &
{\scriptsize Rel. $L^2$-errors on inner edge~$\edge^{(0)}$} 
\end{tabular}
\caption{Resulting relative $L^2$-errors by performing $L^2$ approximation for the function~\eqref{eq:approx_func} on the considered three-patch volume~$\overline{\domain}$ in Section~\ref{subsec:L2approximation}.}
\label{fig:examples}
\end{center}
\end{figure}

%%%%%%%%%%%%%%%%%%%%%%%%%%%%%%%%%%%%%%%%%%%%%%%%%%%%%%%%%%%%%%%%%%%%%
\section{Conclusion} \label{sec:conclusion}
%%%%%%%%%%%%%%%%%%%%%%%%%%%%%%%%%%%%%%%%%%%%%%%%%%%%%%%%%%%%%%%%%%%%%

We explored the concept of the $C^1$ isogeometric spline space~$\V^1$ over trilinearly parameterized multi-patch volumes~$\overline{\domain}$. 

Thereby, the main purpose of the paper was the design of a technique, which allows for a given trilinear multi-patch volume~$\overline{\domain}$ a simple and 
uniform construction of the $C^1$ isogeometric spline space~$\V^1$ and of an associated basis. The proposed procedure is based on the two-patch 
construction~\cite{BiKa19} and can be applied to any spline degree~$p \geq 3$. For the subclass~$\classVolume$ of trilinearly parameterized multi-patch volumes 
with one inner edge, we described the basis construction in more detail, and numerically studied some properties of the resulting $C^1$ isogeometric  
spline space~$\V^1$. We computed on the one hand its dimension, and investigated on the other hand its approximation properties by means of 
$L^2$~approximation.
Our presented construction leads to $C^1$ isogeometric basis functions which are given as the linear combination of explicitly given and locally supported 
functions. The scalar factors for such a linear combination are computed by solving a homogeneous system of linear equations. 
Using e.g. the minimal determining set algorithm described in~\cite{KaVi17b}, the resulting $C^1$ functions are locally supported with respect to the entire 
multi-patch volume, but can possess in the worst case a support over one edge or over the edges containing one vertex. 

A first interesting topic for future research is now the design of 
fully locally supported basis functions, e.g. by enforcing additional smoothness conditions across the edges and vertices similar to the bivariate case 
in~\cite{KaSaTa19a, KaSaTa19b}, where the functions are additionally enforced to be $C^2$ at the vertices.
The paper leaves several further open issues which are worth to study. 
One is the theoretical investigation of the 
numerically obtained results about the properties of the $C^1$ isogeometric spline space such as its dimension and its approximation properties. Further topics of interest are e.g. the generalization of our 
approach to an even wider class of multi-patch volumes than the trilinear ones and the study of possible applications of the constructed $C^1$ isogeometric 
spline functions such as the biharmonic equation, the Cahn-Hilliard equation or problems of strain gradient elasticity.
 
\section*{Acknowledgements} 
M. Kapl has been partially supported by the Austrian Science Fund (FWF) through the project P~33023-N.
V.~Vitrih has been partially supported by the Slovenian Research Agency (research program P1-0404 and research projects J1-9186, J1-1715).
This support is gratefully acknowledged.

\end{document}